\documentclass[11pt]{amsart}
\usepackage{amsthm,enumerate}
\usepackage{amssymb}
\usepackage{color}
\usepackage{mathtools}
\usepackage[bookmarks, bookmarksdepth=2, colorlinks=true, linkcolor=blue, citecolor=blue, urlcolor=blue]{hyperref}

\usepackage[utf8]{inputenc}
\usepackage[T1]{fontenc}
\usepackage{lmodern}
\usepackage{cleveref}
\usepackage{hyperref}
\usepackage{algorithm}
\usepackage{comment}
\usepackage{url, mathrsfs}
\usepackage{amsmath}
\usepackage{graphicx, 
epstopdf}
\AppendGraphicsExtensions{.gif}
\usepackage{subfig}
\usepackage{color}
\usepackage{bbm}
\usepackage{subfloat}
\usepackage[draft]{fixme}
\usepackage{bm} 
\usepackage{bbm} 
\usepackage{epigraph}
\usepackage{endnotes}
\usepackage{ifpdf}
\usepackage{array}
\usepackage{wrapfig}
\usepackage{longtable}
\usepackage{multicol}

\usepackage{multicol}
\usepackage{multirow}
\usepackage{graphicx}
\usepackage{epstopdf}
\usepackage{tikz}
\usepackage{tikz-cd}
\usepackage{makecell}
\usepackage{fullpage}

\usepackage{media9}%

\theoremstyle{plain}
\newtheorem{theorem}{Theorem}[section]
\newtheorem{proposition}[theorem]{Proposition}
\newtheorem{corollary}[theorem]{Corollary}
\newtheorem{lemma}[theorem]{Lemma}

\newtheorem*{claim*}{Claim}

\newtheoremstyle{case}{}{}{}{}{}{:}{ }{}
\theoremstyle{case}

\theoremstyle{definition}
\newtheorem{definition}[theorem]{Definition}
\newtheorem{example}[theorem]{Example}
\newtheorem{remark}[theorem]{Remark}

\newcommand \R{{\mathbb R}}

\newcommand \PP{{\mathbb P}}

\newcommand \K{{\mathcal K}}

\def \inf{\mathrm{inf}}

\DeclareMathOperator{\diag}{Diag}

\DeclareMathOperator{\inter}{int}

\def \L{{\mathcal{L}}}

\title{$\K$-Lorentzian Polynomials, Semipositive Cones, and Cone-Stable EVI Systems}

\author{Papri Dey}
\address{Applied Mathematics, Baskin School of Engineering, University of California, Santa Cruz}
\date{}


\begin{document}
\begin{abstract}
Lorentzian and completely log-concave polynomials have recently emerged as a unifying framework for negative dependence, log-concavity, and convexity phenomena in combinatorics and probability.
We extend this framework to variational analysis and cone-constrained dynamics by studying \emph{$\K$-Lorentzian} and \emph{$\K$-completely log-concave} polynomials over a proper convex cone $\K\subset\R^n$.
For a $\K$-Lorentzian form $f$ and $v\in\operatorname{int}\K$ we associate an open cone $\K^\circ(f,v)$ and a closed cone $\K(f,v)$ via directional derivatives along $v$, recovering the usual hyperbolicity cone when $f$ is hyperbolic.
We prove that $\K^\circ(f,v)$ is a proper cone, equal to the interior of  $\K(f,v)$; if $f$ is $\K(f,v)$-Lorentzian, then $\K(f,v)$ is convex and maximal among convex cones on which $f$ is Lorentzian.
Using the Rayleigh matrix $M_f(x)=\nabla f(x)\nabla f(x)^{\mathsf T}-f(x)\nabla^2 f(x)$, we obtain cone-restricted Rayleigh inequalities and show that two-direction Rayleigh inequalities on $\K$ are equivalent to an acuteness condition on $\K$ for the bilinear form $v^{\mathsf T}M_f(x)w$.
This yields a cone-restricted negative–dependence interpretation linking the curvature of $\log f$ to covariance structures of associated Gibbs measures.
For a nonsingular matrix $A$ we study the determinantal generating polynomial $f_A(x)=\det(\sum_j x_j D_j)$ and show that it is hyperbolic, with hyperbolicity cone intersecting the nonnegative orthant exactly in the classical semipositive cone $K_A={x\ge 0:Ax\ge 0}$.
We generalize this to arbitrary proper cones by introducing $\K$-semipositive cones $K_A=\Lambda_+(f_A,e)\cap\K$ and proving that $f_A$ is $K_A$-Lorentzian.
Finally, for linear evolution variational inequality (LEVI) systems we show that if $q(x)=x^{\mathsf T}Ax$ is (strictly) $\K$-Lorentzian, then $A$ is (strictly) $\K$-copositive and Lyapunov semi-/positive stable on $\K$.
As a consequence, the trivial solution is (asymptotically) cone-stable, giving new Lyapunov criteria for cone-constrained dynamics expressed in terms of $\K$-Lorentzian quadratic forms and their associated cones.

\end{abstract}
\maketitle
\section{Introduction}

Variational inequalities and evolution inclusions on closed convex sets are a
central theme of variational analysis, with applications ranging from contact
and friction in mechanics to constrained control and conic optimization. A
natural and flexible class of such models is given by \emph{evolution
variational inequality} (EVI) systems of the form
\[
  \dot x(t) + Ax(t) + F(x(t)) \in -N_{\K}(x(t)),
  \qquad x(t)\in\K,\ t\ge t_0,
\]
where $A\in\R^{d\times d}$, $F:\R^d\to\R^d$ is nonlinear, $N_{\K}(x)$ is the
normal cone to a closed convex set $\K\subset\R^d$, and $t\mapsto x(t)$ is the
state trajectory. Under standard structural assumptions on $F$, such systems
admit unique global solutions, and a detailed theory of stability with respect
to the constraint set $\K$-\emph{cone-stability} when $\K$ is a cone has
been developed in \cite{Goelevenstability,stabilitylevi,Henrionconecopositive}
and the references therein.

In many applications the relevant constraint set is a proper convex cone
$\K\subset\R^d$, encoding positivity, unilateral constraints, or a conic
feasibility region. Even when the origin is unstable as an equilibrium of the
unconstrained system
\(
  \dot x(t) = -Ax(t) - F(x(t)),
\)
it may become stable once trajectories are constrained to remain in $\K$. This
phenomenon motivates a refined qualitative analysis of cone-constrained
dynamics, Lyapunov functions adapted to $\K$, and their interaction with the
geometry of cone-preserving linear maps and conic optimization.

Parallel to these developments, there has been intensive work on \emph{hyperbolic},
\emph{Lorentzian}, and \emph{completely log-concave} polynomials, originating
in convex algebraic geometry, optimization, and probability. Hyperbolic
polynomials give rise to hyperbolicity cones and hyperbolic barriers in conic
optimization, while Lorentzian and completely log-concave polynomials encode
strong log-concavity and negative dependence properties of discrete
distributions. A key feature of these polynomial classes is that they carry
rich second-order information through Rayleigh-type inequalities and curvature
bounds for their logarithms.

The aim of this paper is to use the $\K$-Lorentzian framework to uncover
new geometric and dynamical structures in variational analysis. We associate to
a $\K$-Lorentzian polynomial a canonical cone $\K(f,v)$ defined by directional
derivatives along $v\in\inter\K$, analyze when this cone is convex and maximal,
and relate it to classical hyperbolicity cones and semipositive cones arising in
conic optimization. This yields a semialgebraic, polynomially defined enlargement
of a given modeling cone $\K$ that is still compatible with Lorentzian/hyperbolic
geometry.

First, we develop $\K$-Lorentzian Lyapunov
functions and Rayleigh-type inequalities that give cone-restricted curvature and
negative-dependence information, and use these to derive Lyapunov criteria for
cone-stability of evolution variational inequality (EVI) systems. Second, we
show how determinantal generating polynomials of semipositive matrices induce
$\K$-semipositive cones that serve simultaneously as invariant cones for
cone-preserving linear maps and as hyperbolic barrier domains for conic
optimization. This connects Lorentzian/hyperbolic polynomial geometry with
cone-stable dynamics, invariant cones, and hyperbolic barrier methods in a
single unified framework.

\subsection*{Main contributions}

We briefly summarize the main results.

\smallskip\noindent
\textbf{1. Cones associated with $\K$-Lorentzian polynomials.}
Let $\K\subset\R^n$ be a proper convex cone and $f\in\R_n^d[x]$ a
$\K$-Lorentzian form. Fix $v\in\operatorname{int}\K$. Motivated by the
hyperbolicity cone of a hyperbolic polynomial, we associate to $(f,v)$ an open
semialgebraic cone
\[
  \K^\circ(f,v)
  = \{x : f(x)>0,\ D_v f(x)>0,\dots,D_v^{d-1}f(x)>0\},
\]
and its closure
\[
  \K(f,v)
  = \{x : f(x)\ge0,\ D_v f(x)\ge0,\dots,D_v^{d-1}f(x)\ge0\}.
\]
We show that $\K^\circ(f,v)$ is a proper cone (contains no line), has
nonempty interior, and that (\Cref{them:cone})
\[
  \operatorname{int}\K(f,v)=\K^\circ(f,v).
\]
When $f$ is hyperbolic, this construction recovers the usual hyperbolicity
cone. We also prove that the ambient cone $\K$ is always contained in
$\K(f,v)$ \Cref{lemma:inclusion}), and that if $f$ is $\K(f,v)$-Lorentzian (equivalently,
$\K(f,v)$-CLC), then $\K(f,v)$ is convex (\Cref{them:convexcone}). In this case $\K(f,v)$ is the
largest convex cone containing $v$ in its interior on which $f$ is Lorentzian.
We discuss examples where $\K(f,v)$ fails to be convex even though $f$ is
Lorentzian, and formulate an open problem on characterizing convex cones of
the form $\K(f,v)$. 

\smallskip\noindent
\textbf{2. Rayleigh differences and cone-restricted negative curvature.}
For a homogeneous polynomial $f$ we introduce the Rayleigh matrix
\[
  M_f(x)
  = \nabla f(x)\,\nabla f(x)^{\mathsf T} - f(x)\,\nabla^2 f(x).
\]
If $f$ is $\K$-Lorentzian, then $-\nabla^2\log f(x)\succeq 0$ on
$\operatorname{int}\K$, so $M_f(x)\succeq 0$ there. This yields diagonal
Rayleigh inequalities
\[
  R_u f(x)
  := (D_u f(x))^2 - f(x)D_u^2 f(x)
  = u^{\mathsf T}M_f(x)u \ge 0,
\]
for all $x\in\K$ and all $u\in\R^n$ (\Cref{them:rayleighdifference}). We then study the two-direction Rayleigh
differences
\[
  R_{v,w} f(x)
  := D_v f(x)\,D_w f(x) - f(x)D_vD_w f(x)
  = v^{\mathsf T}M_f(x)w,
\]
and show that nonnegativity $R_{v,w} f(x)\ge0$ for all $v,w\in\K$ is equivalent
to an \emph{acuteness} condition: the cone $\K$ must be acute with respect to
the bilinear form $\langle v,w\rangle_{M_f(x)}=v^{\mathsf T}M_f(x)w$ (\Cref{prop:two-dir-rayleigh})). In the
polyhedral case this reduces to checking Rayleigh inequalities on generating
rays. This provides a cone-restricted analogue of strong Rayleigh negative
dependence and clarifies the role of $M_f(x)$ as a matrix-valued refinement of
scalar Rayleigh differences, with implications for Gibbs measures and
log-concave sampling (see discussion after \Cref{def:rayleigh_measure}). 

\smallskip\noindent
\textbf{3. Semipositive cones and hyperbolic generating polynomials.}
We then turn to cones arising from cone-preserving linear maps and conic
optimization. For a nonsingular matrix $A\in\R^{n\times n}$ we consider the
determinantal generating polynomial
\[
  f_A(x) = \det\Bigl(\sum_{j=1}^n x_j D_j\Bigr),
\]
where $D_j$ is the diagonal matrix whose diagonal entries are the $j$th column
of $A$. We show that $f_A$ is hyperbolic (\Cref{prop:hyp}) and that its hyperbolicity cone (\Cref{prop:simhypcone})
intersects the nonnegative orthant in the \emph{semipositive cone}
\[
  K_A = \{x\in\R^n_{\ge0}: Ax\ge0\}.
\]
We then generalize semipositive matrices and semipositive cones to an
arbitrary proper convex cone $\tilde\K$, defining $\tilde\K$-semipositive
matrices and associated cones
\(
  K_A = \Lambda_+(f_A,e)\cap\tilde\K.
\)
We prove that $f_A$ is $K_A$-Lorentzian (\Cref{them:clcsemipositive})), thereby exhibiting $f_A$ as a
natural hyperbolic/$\K$-Lorentzian barrier on $K_A$ and linking hyperbolic
geometry with the theory of cone-preserving linear maps.

\smallskip\noindent
\textbf{4. Cone-stability of LEVI systems via $\K$-Lorentzian quadratics.}
In the final part of the paper we apply the above constructions to EVI and, in
particular, to linear EVI (LEVI) systems. We show how these notions yield Lyapunov criteria for
cone-stability of linear evolution variational inequality (LEVI) systems,
providing a bridge between $\K$-Lorentzian quadratic forms and classical
stability theory on cones. We recall abstract Lyapunov criteria for stability and asymptotic stability with respect to a cone $\K$,
expressed in terms of $\K$-copositive matrices and Lyapunov semi/positive
stability.

We then show that if the quadratic form $f(x)=x^{\mathsf T}Ax$ is
(respectively, strictly) $\K$-Lorentzian, then $A$ is (respectively, strictly)
$\K$-copositive and hence Lyapunov semi/positive stable on $\K$. As a
consequence, the trivial solution of the LEVI system is (asymptotically)
stable with respect to $\K$ (\Cref{prop:clcconstable}). More generally, if symmetric part of $A$ is $\K$-copositive and has exactly one positive eigenvalue, by \Cref{prop:levi-hyperbolic-cone-stability} the trivial solution of the LEVI system is (asymptotically) stable w.r.t $\K$. Furthermore, if a quadratic $f(x)$ is $\K$-Lorentzian, the trivial solution of the class of LEVI systems is (asymptotically) stable w.r.t $\K$ by \Cref{them:levi-hyperbolic-cone-stability}. We illustrate how unstable linear systems in the full space can become asymptotically stable once constrained to cones of the form $\K(f,v)\cap\R^n_{\ge0}$ constructed from $\K$-Lorentzian quadratic forms (\Cref{ex:levi}).


\noindent In \Cref{sec:cone} we construct the cones $\K^\circ(f,v)$ and $\K(f,v)$
associated with a $\K$-Lorentzian polynomial, establish their basic
properties, and relate them to hyperbolicity cones. In \Cref{sec:rayleigh} we
develop cone-restricted Rayleigh inequalities, introduce the Rayleigh matrix
$M_f(x)$, and connect acuteness of $\K$ to two-direction Rayleigh differences
and negative dependence. In \Cref{sec:semipositive} we also introduce
semipositive and $\K$-semipositive cones and show that their generating
polynomials are $\K$-Lorentzian, with applications to conic optimization and
cone-preserving linear maps. Finally, in \Cref{sec:dynamic} we revisit
stability theory for EVI and LEVI systems and derive new Lyapunov criteria in terms of $\K$-Lorentzian quadratic forms, illustrating cone-stability on
explicit low-dimensional examples.
\section{Cone associated with \texorpdfstring{$\K$}{}-Lorentzian polynomials} 
\label{sec:cone}
We now explain how, given a $K$-Lorentzian polynomial, one can canonically
associate to it a proper cone determined by a distinguished interior point.
Let $\R[x]$ represent the space of $n$-variate polynomials over $\R$, and $\R[x]_{\leq d}$ represent the space of $n$-variate polynomials over $\R$ with degree at most $d$, and $\R[x]^d_n$ denote the set of real homogeneous polynomials (aka forms) in $n$ variables of degree $d$. Throughout this subsection, let $\K \subseteq \R^n$ be a proper convex cone and
let $f \in \R[x]_n^d$ be a $\K$-Lorentzian form. 

For a point $a\in \R^n$ and $f\in \R[x]$, $D_a f$ denotes the directional derivative of $f$ in direction $a$: $D_a f=\sum_{i=1}^n a_i\frac{\partial f}{\partial x_i}$. Here are the definitions of the $\K$-completely log-concave polynomials and $\K$-Lorentzian forms.
\begin{definition} \cite{GPlorentzian} \label{def:clc}
A polynomial (form) $f \in \R[x]_{\leq d}$ is called a $\K$-completely log-concave aka $\K$-CLC (form) on a proper convex cone $\K$ if for any choice of $a_{1}, \dots, a_{m} \in \K$, with $m\leq d$, we have that $D_{a_{1}} \dots D_{a_m} f$ is log-concave on $\inter \K$.  
A polynomial (form) $f \in \R[x]$ is strictly $\K$-CLC if for any choice of $a_{1}, \dots, a_{m} \in \K$, with $m\leq d$,  $D_{a_{1}} \dots D_{a_m} f$ is strictly log-concave on all points of $\K$. 
\end{definition}
\begin{definition} \cite{GPlorentzian} \label{def:pls}
Let $\K$ be a proper convex cone. A form $f \in \R[x]_n^d$ of degree $d\geq 2$ is said to be $\K$-Lorentzian if for any $a_{1},\dots,a_{d-2} \in  \inter \K$,  the quadratic form $q=D_{a_{1}} \dots D_{a_{d-2}}f$ satisfies the following conditions:  
\begin{enumerate}
 \item  The matrix $Q$ of $q$ has exactly one positive eigenvalue.
\item  For any $x,y \in\inter \K$ we have $y^t Q x=\langle y,Qx\rangle>0$.
\end{enumerate}
For degree $d \leq 1$ a form is $\K$-Lorentzian if it is nonnegative on $\K$.  
\end{definition} 
Recall that, by \cite[Theorem~4.10]{GPlorentzian}, the classes
of $K$-Lorentzian and $K$-completely log-concave (K-CLC) forms coincide.

In analogy with the hyperbolicity cone of a hyperbolic polynomial with respect
to a direction $e$, we associate to the pair $(f,v)$ a proper cone
$\K(f,v) \subseteq \R^n$ that contains $v$ in its interior. We begin by giving a
polynomial inequality description of its interior, which will be a semialgebraic
set. Since $f$ is $\K$-CLC and $v \in \operatorname{int} \K$, all directional
derivatives $D_v^k f$ are log-concave on $\operatorname{int} \K$. 
\begin{proposition} \label{prop:noncoeffs}
   Let $f(x)$ be a nonzero $\K$-CLC over a proper convex cone $\K$. Then for any $x,v \in \inter \K$,  the coefficients of $f(x+tv)$ are positive. 
   \end{proposition}
   \begin{proof} Since $f$ is CLC over $\K$, its directional derivatives are log-concave, i.e., they satisfy the positivity condition at each degree level $d \geq 1$. Therefore, in particular, $D_{v}f(x+tv)=f^{'}(x+tv) >0,$ and $D_{v}D_{v}f(x+tv)=f^{''}(x+tv) >0$ for all $v \in \inter \K$. 
   That enforces all the coefficients of $f(x+tv)=f(x)+t D_v f(x) +\frac{t^{2}}{2} D_{v}^{2}f(x) +\dots+\frac{t^{d}}{d!}D_{v}^{d}f(x)$ to be positive. 
 \end{proof}  
In particular, 
by \Cref{prop:noncoeffs}, for every $x \in \operatorname{int} \K$ the
coefficients of the univariate restriction
\[
  t \longmapsto f(x + t v)
\]
are positive. This motivates the definition
\[
  \K^\circ(f,v)
  \coloneqq
  \{ x \in \R^n : f(x) > 0,\ D_v f(x) > 0,\ \dots,\ D_v^{d-1} f(x) > 0 \}.
\]
By construction, $\K^\circ(f,v)$ is an open semialgebraic cone, and from
Proposition~\ref{prop:noncoeffs} we also have $v \in \K^\circ(f,v)$.

\begin{proposition}\label{prop:Kcirc_proper}
Let $f \in \R[x]_n^d$ be a homogeneous polynomial of degree $d\ge 1$ and
$v\in\R^n$. Define
\[
\K^\circ(f,v)
:=\bigl\{x\in\R^n : f(x)>0,\;D_v f(x)>0,\;\dots,\;D_v^{d-1}f(x)>0\bigr\}.
\]
Then $\K^\circ(f,v)$ is an open cone, and for every $x\in \K^\circ(f,v)$ we have
$-x\notin \K^\circ(f,v)$. In particular, $\K^\circ(f,v)$ contains no line and hence
is a proper cone (in the sense that its lineality space is trivial).
\end{proposition}

\begin{proof}
Since $f$ is homogeneous of degree $d$, we have $f(\lambda x)=\lambda^d f(x)$ for all
$\lambda>0$. Similarly, for each $k\in\{1,\dots,d-1\}$,
\[
D_v^k f(\lambda x) = \lambda^{d-k} D_v^k f(x),
\]
because $D_v^k f$ is homogeneous of degree $d-k$. If $x\in \K^\circ(f,v)$ and
$\lambda>0$, then $f(\lambda x)$ and all $D_v^k f(\lambda x)$ remain strictly
positive. Hence $\lambda x\in \K^\circ(f,v)$, so $\K^\circ(f,v)$ is a cone (closed
under multiplication by positive scalars). Openness follows directly from the
strict inequalities defining $\K^\circ(f,v)$.

Next, we show that $\K^\circ(f,v)$ does not contain any pair $\{\pm x\}$ with
$x\neq 0$. Using homogeneity and the fact that each $k$th order derivative
$D_v^k f$ is homogeneous of degree $d-k$, one checks that for each
$k\in\{0,1,\dots,d-1\}$,
\[
D_v^k f(-x) = (-1)^{d-k} D_v^k f(x).
\]
Indeed, this holds for $k=0$ by $f(-x)=(-1)^d f(x)$, and each differentiation
with respect to $v$ reduces the degree by one, introducing a factor $-1$.

Now fix $x\in \K^\circ(f,v)$. Then $D_v^k f(x)>0$ for all $k=0,\dots,d-1$. If
$d$ is even, then for $k=1$ we have
\[
D_v f(-x) = (-1)^{d-1} D_v f(x) = - D_v f(x) <0,
\]
so $-x\notin \K^\circ(f,v)$. If $d$ is odd, then for $k=0$ we have
\[
f(-x) = (-1)^d f(x) = - f(x) <0,
\]
so again $-x\notin \K^\circ(f,v)$. Thus in all cases $x\in \K^\circ(f,v)$ implies
$-x\notin \K^\circ(f,v)$.

Since $\K^\circ(f,v)$ is a cone (closed under $\lambda>0$) and contains no pair
$\{\pm x\}$ with $x\neq 0$, it cannot contain any line $\{x_0 + t w : t\in\R\}$:
if it did, then by conic property that line would pass through the origin and contain
some $w\neq 0$ with both $w$ and $-w$ in $\K^\circ(f,v)$, contradicting the
previous paragraph. Hence $\K^\circ(f,v)$ contains no line, so its lineality
space is trivial. In particular,
it is a cone with nonempty interior (since $v\in \K^\circ(f,v)$) and contains no line. This is precisely the statement that $\K^\circ(f,v)$ is a
proper cone.
\end{proof}

\begin{proposition}\label{prop:connected}
Let $f \in \R[x]_n^d$ be a homogeneous polynomial of degree $d\ge 1$ and fix
$v \in \R^n$. Let $\K^\circ(f,v)$ be as above. Then for every $x \in \K^\circ(f,v)$
there exists a continuous path from $x$ to $v$ along which $f$ remains strictly
positive. In particular, $\K^\circ(f,v)$ is contained in the connected component
of $\{x \in \R^n : f(x) \neq 0\}$ that contains $v$.
\end{proposition}

\begin{proof}
Fix $x \in \K^\circ(f,v)$ and let $l$ be the line segment with endpoints $x$ and $v$,
that is
\[
  l \coloneqq \{ y = (1-s)x + s v : s \in [0,1] \}.
\]

\medskip\noindent
\emph{Step 1: positivity along the rays from $x$ and $v$.}
For $w \in \R^n$ consider the univariate polynomial
\[
  g_w(t) \coloneqq f(w + t v), \qquad t \in \R.
\]
By repeated differentiation and homogeneity of $f$, we have the Taylor expansion
\[
  g_w(t)
  = f(w+tv)
  = \sum_{k=0}^d \frac{1}{k!} D_v^k f(w) \, t^k.
\]
If $w \in \K^\circ(f,v)$, then by definition $D_v^k f(w) > 0$ for all $k=0,\dots,d-1$,
and we also have $D_v^d f(w) > 0$ (since $f$ is homogeneous, $D_v^d f$ is a constant
independent of $w$). Hence all coefficients of $g_w$ are positive, and therefore
$g_w(t) > 0$ for all $t \ge 0$.

Applying this with $w=x$ and $w=v$, we conclude that
\[
  f(x+tv) > 0 \quad\text{and}\quad f(v+tv) > 0 \qquad \text{for all } t \ge 0.
\]
In particular, the rays $\{x+tv : t \ge 0\}$ and $\{v+tv : t \ge 0\}$ are contained in
the positive set $\{f>0\}$.

\medskip\noindent
\emph{Step 2: positivity on a translated segment.}
For each $y \in l$, consider again $h_y(t) \coloneqq f(y+tv)$. As above,
\[
  h_y(t)
  = \sum_{k=0}^d \frac{1}{k!} D_v^k f(y) \, t^k,
\]
and the leading coefficient is
\[
  \frac{1}{d!} D_v^d f(y).
\]
Since $f$ is homogeneous of degree $d$ and $D_v^d$ lowers the degree by $d$, the
polynomial $D_v^d f$ is homogeneous of degree $0$, hence constant in $y$. From
$v \in \K^\circ(f,v)$ we know $D_v^d f(v) > 0$, so
\[
  D_v^d f(y) = D_v^d f(v) > 0 \qquad \text{for all } y \in \R^n.
\]
In particular, for each fixed $y\in l$ the univariate polynomial $h_y(t)$ has positive
leading coefficient, so $h_y(t) \to +\infty$ as $t \to +\infty$. Thus, for every
$y \in l$ there exists some $\lambda_y > 0$ such that $h_y(t) > 0$ for all
$t \ge \lambda_y$.

The map $(y,t) \mapsto h_y(t)$ is continuous on $l \times [0,\infty)$, and $l$ is
compact. Therefore we can choose a uniform $\lambda > 0$ such that
\[
  f(y + \lambda v) = h_y(\lambda) > 0 \qquad \text{for all } y \in l.
\]
Hence the translated segment $l + \lambda v = \{ y + \lambda v : y \in l \}$ is
contained in the positive set $\{f>0\}$.

\medskip\noindent
\emph{Step 3: concatenating paths.}
Combining the two steps above, we obtain a continuous path from $x$ to $v$ along
which $f$ remains strictly positive:
\[
  \{x + t v : 0 \le t \le \lambda\}
  \;\cup\;
  \{y + \lambda v : y \in l\}
  \;\cup\;
  \{v + t v : \lambda \ge t \ge 0\}.
\]
This path lies entirely in $\{f>0\} \subseteq \{f \neq 0\}$. Therefore $x$ and $v$
belong to the same connected component of $\{x : f(x) \neq 0\}$, and since
$x \in \K^\circ(f,v)$ was arbitrary, the set $\K^\circ(f,v)$ is contained in the
connected component of $\{f \neq 0\}$ that contains $v$.
\end{proof}
\begin{remark}
One might hope that $\K^\circ(f,v)$ coincides with the entire connected component
of $\{x : f(x)\neq 0\}$ containing $v$. This is false in general. Our results only
require the containment $\K^\circ(f,v)$ in the positive component of $\{f\neq 0\}$,
not equality.
\end{remark}

\noindent We now define the closed cone associated with $(f,v)$ by
\begin{equation}\label{eq:cone}
  \K(f,v) \coloneqq \bigl\{x \in \R^{n} : f(x) \geq 0,\ D_{v}f(x) \geq 0,\ \dots,\ D_{v}^{d-1}f(x) \geq 0 \bigr\}.
\end{equation}
Thus $\K(f,v)$ is a closed semialgebraic cone, being a finite intersection of sets
of the form $\{x : p(x) \ge 0\}$ with $p$ polynomial. We will show that $\K(f,v)$ is precisely the (topological) closure of the open cone $\K^\circ(f,v)$ and that $\K^\circ(f,v)$ coincides with the interior of $\K(f,v)$.
\begin{lemma}\label{lem:taylor-sign}
Let $g:\R\to\R$ be a real-analytic function and let $t_0\in\R$.
Suppose there exists an integer $\ell\ge 1$ such that
\[
  g^{(j)}(t_0) = 0 \quad \text{for all } j=0,1,\dots,\ell-1,
  \qquad \text{and} \qquad
  g^{(\ell)}(t_0) \neq 0.
\]
Then there exists $\delta>0$ such that for all $|t-t_0|<\delta$,
\begin{equation}\label{eq:taylor-leading}
  g(t)
  = \frac{g^{(\ell)}(t_0)}{\ell!}\,(t-t_0)^\ell + R(t),
\end{equation}
where $R(t)$ satisfies
\[
  \lim_{t\to t_0} \frac{R(t)}{(t-t_0)^\ell} = 0.
\]
In particular:
\begin{enumerate}[(a)]
  \item If $\ell$ is odd, then for every neighborhood of $t_0$ the function $g$
  takes both positive and negative values; more precisely, for $|t-t_0|$ small
  enough,
  \[
    \operatorname{sign} g(t)
    = \operatorname{sign}\bigl(g^{(\ell)}(t_0)\bigr)\cdot \operatorname{sign}(t-t_0).
  \]
  \item If $\ell$ is even, then for $|t-t_0|$ small enough,
  \[
    \operatorname{sign} g(t)
    = \operatorname{sign}\bigl(g^{(\ell)}(t_0)\bigr).
  \]
  In particular, if $g^{(\ell)}(t_0) < 0$, then $g(t)<0$ for all
  $0<|t-t_0|<\delta$, so $t_0$ is not an interior point of the superlevel set
  $\{t : g(t)\ge 0\}$.
\end{enumerate}
\end{lemma}

\begin{proof}
By Taylor's theorem at $t_0$ we have
\[
  g(t)
    = \sum_{j=0}^{\ell-1} \frac{g^{(j)}(t_0)}{j!}(t-t_0)^j
      + \frac{g^{(\ell)}(t_0)}{\ell!}(t-t_0)^\ell + R(t),
\]
where $R(t)$ satisfies $R(t) = o(|t-t_0|^\ell)$ as $t\to t_0$. By hypothesis,
$g^{(j)}(t_0)=0$ for $j=0,\dots,\ell-1$, so the sum over $j$ vanishes and
\eqref{eq:taylor-leading} holds. The limit property of $R(t)$ implies that,
for $|t-t_0|$ small enough, the sign of $g(t)$ is the same as the sign of
$\frac{g^{(\ell)}(t_0)}{\ell!}(t-t_0)^\ell$. If $\ell$ is odd, the factor
$(t-t_0)^\ell$ changes sign as $t$ passes through $t_0$, whereas if $\ell$
is even it does not. This yields the two cases.
\end{proof}

\begin{theorem}\label{them:cone}
Let $\K \subseteq \R^n$ be a proper convex cone and let $f$ be a $\K$-Lorentzian
polynomial. Fix $v \in \K^\circ(f,v)$ (in particular $v\in\operatorname{int} \K$).
Then $\K^\circ(f,v)$ is the interior of $\K(f,v)$. 
\end{theorem}

\begin{proof}
By definition we have $\K^\circ(f,v) \subset \K(f,v)$, $\K^\circ(f,v)$ is open, and $\K(f,v)$ is closed. Hence
\[
  \overline{\K^\circ(f,v)} \subseteq \K(f,v).
\]
It remains to show that $\K^\circ(f,v)$ is the interior of $\K(f,v)$, i.e.\ that
every point of $\K(f,v)\setminus \K^\circ(f,v)$ lies on the boundary of $\K(f,v)$.

Let $x_0 \in \K(f,v) \setminus \K^\circ(f,v)$. Then at least one of
\[
  f(x_0),\ D_v f(x_0),\ \dots,\ D_v^{d-1} f(x_0)
\]
is equal to zero. Consider the univariate polynomial
\[
  h(t) \coloneqq f(x_0 + t v), \qquad t \in \R.
\]
Its Taylor expansion at $t=0$ is
\[
  h(t) = \sum_{k=0}^d \frac{1}{k!} D_v^k f(x_0)\, t^k.
\]

\medskip\noindent
\emph{Case~1: $D_v^k f(x_0) = 0$ for all $k=0,1,\dots,d-1$.}
Then all derivatives $h^{(k)}(0)$ vanish for $k=0,\dots,d-1$, so $h$ must be the
zero polynomial. Hence $f(x_0 + t v) = 0$ for all $t\in\R$. Since $v \in
\K^\circ(f,v)$ and $\K$ is proper, we can choose $t>0$ small enough so that
$x_0 + t v \in \operatorname{int} \K$. But $f>0$ on $\operatorname{int}\K$ for a
$\K$-Lorentzian polynomial, so this is impossible. Thus Case~1 cannot occur.

\medskip\noindent
\emph{Case II: $x_0$ is a common root of $f$ and $D_v f$, but not all derivatives
vanish.} Suppose $x_0 \in K(f,v)\setminus K^\circ(f,v)$ satisfies
\[
  f(x_0) = 0, \qquad D_v f(x_0) = 0,
\]
and that not all derivatives $D_v^k f(x_0)$ vanish for $k=2,\dots,d-1$. Let
$m \in \{2,\dots,d-1\}$ be the smallest index such that $D_v^m f(x_0) \neq 0$.
Consider the univariate polynomial
\[
  g(t) \coloneqq D_v f(x_0 + t v), \qquad t \in \R.
\]
Then
\[
  g(0) = D_v f(x_0) = 0, \qquad g^{(j)}(0) = D_v^{j+1} f(x_0)
  \quad\text{for all } j\ge 0.
\]
By choice of $m$, we have
\[
  g^{(j)}(0) = 0 \ \text{for } j=0,1,\dots,m-2,
  \qquad\text{and}\qquad
  g^{(m-1)}(0) = D_v^m f(x_0) \neq 0.
\]
Thus we are in the setting of Lemma~\ref{lem:taylor-sign} with $\ell = m-1$.
If $m-1$ is odd, the lemma implies that $g(t)$ changes sign in every
neighborhood of $t=0$, i.e.\ $D_v f(x_0+tv)$ takes negative values for $t$
arbitrarily close to $0$. This contradicts the assumption that $x_0$ is an
interior point of the superlevel set $\{x : D_v f(x)\ge 0\}$.

If $m-1$ is even and $g^{(m-1)}(0) = D_v^m f(x_0) < 0$, then by
Lemma~\ref{lem:taylor-sign} there exists $\delta>0$ such that
$g(t) < 0$ for all $0<|t|<\delta$, again contradicting interiority with respect
to the constraint $D_v f\ge 0$.

Finally, if $m-1$ is even and $D_v^m f(x_0) > 0$, then $g(t)>0$ for all
$0<|t|<\delta$ for some $\delta>0$, so the constraint $D_v f\ge 0$ alone does
not prevent $x_0$ from being interior. However, by definition of $K(f,v)$,
all higher directional derivatives $D_v^k f$ for $k\ge 2$ also impose
nonnegativity constraints. Repeating the same argument with the first $k$ for
which $D_v^k f(x_0)=0$ and applying Lemma~\ref{lem:taylor-sign} to the
univariate restriction
\[
  t \longmapsto D_v^k f(x_0 + t v),
\]
we obtain that for at least one of these derivative constraints the
corresponding superlevel set $\{x : D_v^k f(x)\ge 0\}$ has a genuine boundary
at $x_0$ (either by a sign change when the first nonzero derivative has odd
order, or by local negativity when it has even order and negative coefficient).
In all cases, there exist points arbitrarily close to $x_0$ violating at least
one of the defining inequalities of $K(f,v)$, so $x_0$ cannot be an interior
point of $K(f,v)$.

In either case, $x_0$ cannot be an interior point of $\K(f,v)$, so
$\K(f,v)\setminus \K^\circ(f,v)$ consists entirely of boundary points. Therefore
\[
  \operatorname{int}\K(f,v) = \K^\circ(f,v).
\]
\end{proof}

\begin{remark}[Closure of $\K^\circ(f,v)$]
By definition we have $\K^\circ(f,v)\subseteq \K(f,v)$, with $\K^\circ(f,v)$ open
and $\K(f,v)$ closed. Hence
\[
  \overline{\K^\circ(f,v)} \;\subseteq\; \K(f,v).
\]
In general, for an arbitrary closed semialgebraic set $S$, one cannot conclude
that $S$ is the closure of its interior: closed semialgebraic sets may contain
lower--dimensional components that are disjoint from $\overline{\operatorname{int}S}$.
Thus our arguments only guarantee the inclusion
$\overline{\K^\circ(f,v)} \subseteq \K(f,v)$ in full generality.

In the hyperbolic case, however, equality is known to hold. If $f$ is
hyperbolic with respect to $v$ and $f(v)>0$, then (see Renegar~\cite{Renegar}
and Saunderson--Parrilo~\cite{saundersonhyp})
\[
  \K^\circ(f,v) = \Lambda_{++}(f,v),
  \qquad
  \K(f,v) = \Lambda_{+}(f,v),
\]
where \begin{align*}
\Lambda_{++}(f,v)
  &= \{x \in \R^n : \text{all zeros of } t \mapsto f(x+tv) \text{ are real and negative}\} \\
  &= \{x \in \R^n : f(x+tv) = 0 \implies t < 0\},
\end{align*}
and 
\begin{align*}
\Lambda_{+}(f,v)
  &= \{x \in \R^n : \text{all zeros of } t \mapsto f(x+tv) \text{ are real and nonpositive}\} \\
  &= \{x \in \R^n : f(x+tv) = 0 \implies t \le 0\}.
\end{align*}
denote the open and closed
hyperbolicity cones of $f$ with respect to $v$ respectively. In particular,
$\Lambda_{+}(f,v) = \overline{\Lambda_{++}(f,v)}$, so in this case
\[
  \K(f,v) = \overline{\K^\circ(f,v)}.
\]
It is an open question in our setting to characterize those $\K$-Lorentzian
polynomials $f$ for which this equality holds beyond the hyperbolic case.
\end{remark}

It is natural to ask whether $\K(f,v)$ is convex, like a hyperbolicity cone,
when $f$ is $\K$-Lorentzian but not hyperbolic. In \Cref{ex:counter} we show that
for a Lorentzian but nonhyperbolic polynomial $f$, the cone $K(f,v)$ need not
be convex, and in that example the nonnegative orthant is contained in $\K(f,v)$.

\begin{lemma}\label{lemma:inclusion}
Let $\K$ be a proper convex cone with $v \in \operatorname{int}\K$, and let $f$ be
a $\K$-Lorentzian polynomial. Then $\K \subseteq \K(f,v)$.
\end{lemma}

\begin{proof}
By the equivalence between $\K$-Lorentzian and $\K$-completely log-concave
forms and
Proposition~\ref{prop:noncoeffs}, for every $x \in \operatorname{int}K$ we have
\[
  f(x) > 0,\quad D_v f(x) > 0,\quad \dots,\quad D_v^{d-1} f(x) > 0,
\]
so $\operatorname{int}\K \subseteq \K^\circ(f,v) \subseteq \K(f,v)$. Since
$\K(f,v)$ is closed and $\operatorname{int}\K$ is dense in $\K$, it follows that
$\K \subseteq \K(f,v)$.
\end{proof}

\begin{theorem}\label{them:convexcone}
Let $\K$ be a proper convex cone containing $v$ in its interior, and let $f$ be a $\K$-Lorentzian polynomial. Consider
\[
  \K(f,v) = \{x \in \R^{n}: f(x) \geq 0, D_{v}f(x) \geq 0, \dots, D_{v}^{d-1}f(x) \geq 0\}.
\]
If $f$ is $\K(f,v)$-Lorentzian with respect to $v$ (equivalently, $\K(f,v)$-CLC), then $\K(f,v)$ is convex.
\end{theorem}

\begin{proof}
Since $f$ is $\K(f,v)$-Lorentzian, each function
\[
  g_k(x) \coloneqq D_v^k f(x), \qquad k=0,1,\dots,d-1,
\]
is log-concave on $\K(f,v)$. In particular, for any $x_1,x_2\in \K(f,v)$ and
$\lambda\in[0,1]$,
\[
  g_k(\lambda x_1 + (1-\lambda)x_2)
    \;\ge\; g_k(x_1)^\lambda\, g_k(x_2)^{1-\lambda}
    \;\ge\; 0,
\]
because $g_k(x_1),g_k(x_2)\ge 0$ by the definition of $\K(f,v)$. Thus each
superlevel set $\{x : g_k(x)\ge 0\}$ is convex. Since $\K(f,v)$ is the
intersection of these convex sets over $k=0,\dots,d-1$, it is itself convex.
\end{proof}

\begin{remark}
By \Cref{lemma:inclusion}, any proper convex cone $\K$ with
$v \in \operatorname{int}\K$ on which $f$ is $\K$-Lorentzian is contained in
$\K(f,v)$. In particular, if $\K(f,v)$ is convex, then among all convex cones containing $v$ in their interior on which $f$ is Lorentzian, $\K(f,v)$ is the largest with respect to inclusion.
\end{remark}
\noindent Note that convexity of $\K(f,v)$ imposes a simple necessary condition: Express
\[
  \K(f,v)
  = \bigcap_{k=0}^{d-1} S_k,
  \qquad
  S_k := \{x \in \R^n : D_v^k f(x) \ge 0\},
\]
we see that if $\K(f,v)$ is convex and nonempty, then each $S_k$ must be convex.
Equivalently, each directional derivative $g_k(x) := D_v^k f(x)$ is
\emph{quasi-concave} on $\K(f,v)$ in the sense that
$g_k(\lambda x_1 + (1-\lambda)x_2) \ge \min\{g_k(x_1),g_k(x_2)\}$ for all
$x_1,x_2 \in \K(f,v)$ and $\lambda \in [0,1]$. Our sufficient condition in
Theorem~\ref{them:convexcone} (namely that $f$ is $\K(f,v)$-Lorentzian, so that
each $D_v^k f$ is log-concave on $\K(f,v)$) is strictly stronger than this
necessary quasi-concavity requirement, and thus a naive converse to
Theorem~\ref{them:convexcone} is not expected to hold in general. This leads to
the following open problem: characterize those semialgebraic proper convex cones
$\K \subset \R^n$ for which there exist a polynomial $f$ and a direction
$v \in \operatorname{int}\K$ such that $f$ is $\K$-Lorentzian and $\K = \K(f,v)$.
There are several positive examples. Spectrahedral cones are defined by
determinantal polynomials, and hyperbolicity cones are defined by hyperbolic
polynomials; these polynomials are $\K$-Lorentzian for their associated cones.
Thus classical hyperbolic programming fits into our framework, and suggests the possibility of a broader ``$\K$-Lorentzian programming'' theory extending hyperbolic programming. 

The discussion above suggests that cones of the form $\K(f,v)$ are highly
structured. If $\K = \K(f,v)$ for some $\K$-Lorentzian polynomial $f$, then on the interior, $\inter \K$, the functions $D_v^k f$ ($k=0,\dots,d-1$) form a tower of positive, log-concave homogeneous functions whose superlevel sets are convex and whose zero sets cover the algebraic boundary of $\K$. Thus any convex cone $\K$ that
admits such a ``Lorentzian barrier tower'' is a natural candidate for arising as $\K(f,v)$.

At present we only have partial necessary conditions (e.g., quasi-concavity of the directional derivatives along $v$) and strong sufficient conditions (such as $\K$-Lorentzianity of $f$), and a general characterization of cones of the form $\K(f,v)$ remains open.

\begin{example}\label{ex:counter}
Let
\[
f(x,y,z)=z\bigl(4x^3+15x^2y+18xy^2+6y^3\bigr)
=4x^3z+15x^2yz+18xy^2z+6y^3z,
\qquad v=(1,1,1).
\]
Therefore,
\begin{equation}\label{eq:Kfv-explicit}
\K(f,v)=\Bigl\{(x,y,z)\in\R^3:\ 
\begin{array}{l}
z\bigl(4x^3+15x^2y+18xy^2+6y^3\bigr)\ge 0,\\[2pt]
4x^3+15x^2y+18xy^2+6y^3+27zx^2+66zxy+36zy^2\ge 0,\\[2pt]
54x^2+132xy+72y^2+120zx+138zy\ge 0,\\[2pt]
360x+414y+258z\ge 0
\end{array}
\Bigr\}.
\end{equation}

\medskip
\noindent
On the affine slice $x+y+z=1$, by a uniform random sampling search in MATLAB, we found points
\[
p=(1.2575,\,-0.7975,\,0.54),\qquad
q=(0.2375,\,-0.14125,\,0.90375),
\]
with $p,q\in \K(f,v)$ but midpoint
\[
m=\tfrac12(p+q)=(0.7475,\,-0.469375,\,0.721875)
\]
violating the constraint:
\[
D_v f(m)\approx -1.975\times 10^{-2}<0.
\]
Hence $m\notin\K(f,v)$ and $\K(f,v)$ is nonconvex, already visible on the
2D slice (restricted view range) plot as a ``kinked" region shown in \Cref{fig:non_convex_cone}. 
\begin{figure}[htbp] 
  \centering
    \includegraphics[width=\textwidth]{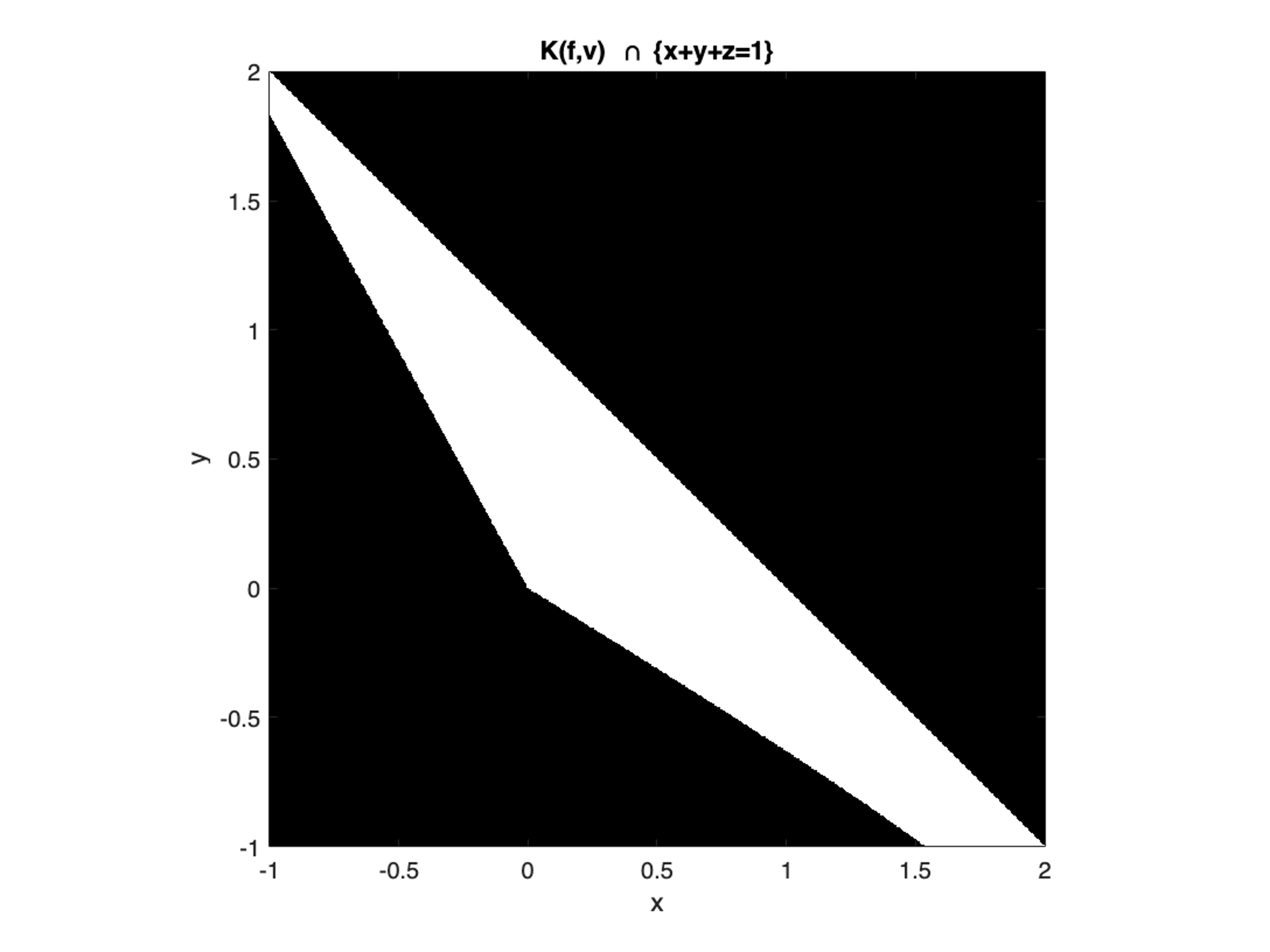}
    \caption{$\K(f,v)\cap\{x+y+z=1\}$. White indicates feasibility.}
    \label{fig:non_convex_cone}
    \end{figure}
Therefore $\K(f,v)\cap H$ is not convex, hence
$\K(f,v)$ is not convex (at least on the connected component containing $p$ and $q$).

\end{example}

\section{Rayleigh Differences and $\K$-Semipositive Cones} \label{sec:rayleigh}

Let $f$ be a multi-affine polynomial. Then $f$ is real stable if and only if
the Rayleigh difference polynomials
\[
  \Delta_{ij} f(x)
  \;\coloneqq\;
  \frac{\partial f}{\partial x_i}(x)\,
  \frac{\partial f}{\partial x_j}(x)
  - f(x)\,
    \frac{\partial^2 f}{\partial x_i\partial x_j}(x)
\]
are nonnegative for all $x\in\R^n$ and all $i,j\in[n]$
\cite[Theorem~5.6]{brandenhalf}. Here we record a basic identity relating
Rayleigh differences and the Hessian of $\log f$.

\begin{lemma}\label{lemma:logmultiaffine-correct}
Let $f \in C^{2}(K;\R)$ with $f(x) > 0$ for all $x \in K$.
Then for every $x\in K$ and every $i,j\in[n]$,
\[
  \Delta_{ij} f(x)
  := \frac{\partial f}{\partial x_i}(x)\,\frac{\partial f}{\partial x_j}(x)
     - f(x)\,\frac{\partial^2 f}{\partial x_i\partial x_j}(x)
  = -\,f(x)^2\,\frac{\partial^2}{\partial x_i\partial x_j}\log f(x).
\]
In particular,
\[
  \Delta_{ij} f(x) \ge 0
  \quad\Longleftrightarrow\quad
  \frac{\partial^2}{\partial x_i\partial x_j}\log f(x)\le 0.
\]
If $f$ is log--concave on $\K$, then the Hessian $\nabla^2 \log f(x)$ is
negative semidefinite for every $x\in \K$, which is strictly stronger than the
entrywise conditions
$\frac{\partial^2}{\partial x_i\partial x_j}\log f(x)\le 0$ (and hence
stronger than $\Delta_{ij}f(x)\ge 0$) for all $i,j$.
\end{lemma}

More generally, for $f\in\R_n^d[x]$ and directions $v,w\in\R^n$ we write
\[
  \Delta_{v,w} f(x)
  \;\coloneqq\;
  D_v f(x)\,D_w f(x) - f(x)\,D_v D_w f(x),
\]
the Rayleigh difference polynomial in directions $v,w$, of degree $2d-2$ in
$x$. For hyperbolic polynomials these Rayleigh polynomials are globally
nonnegative. For instance, consider
\[
  f(x_1,x_2) = -2x_1^3 + 12x_1^2x_2 + 18x_1x_2^2 - 8x_2^3,
\]
which is hyperbolic with respect to $(1,1)$. The Rayleigh difference
$\Delta_{v,w} f$ for suitable $v,w$ is a quartic bivariate polynomial. In this
binary quartic setting, global nonnegativity of $\Delta_{v,w} f$ is equivalent
to a sum-of-squares representation
\[
  \Delta_{v,w} f(x)
  =
  q(x)^{\mathsf T} M\,q(x),
  \qquad
  q(x) = \begin{bmatrix} x_1^2 & x_1x_2 & x_2^2 \end{bmatrix},
\]
with $M\succeq 0$, and this viewpoint underlies semidefinite descriptions of
hyperbolicity cones via nonnegative Rayleigh polynomials; see
\cite{CynthiaSOS}.

For $\K$-Lorentzian polynomials the situation is subtler. Let
$f\in\R_n^d[x]$ be $\K$-Lorentzian over a proper convex cone $\K$ and
$v\in\operatorname{int}\K$. In general, the Rayleigh difference
$\Delta_{v,w} f(x)$ need not be nonnegative on all of $\R^n$, so a global
sum-of-squares representation need not exist. However, it can still be
nonnegative when restricted to the cone. For example, in
\Cref{ex:counter} the Rayleigh difference with respect to $v=(1,1)$ and
$w=(2,1)$ is
\[
  \Delta_{v,w} f(x)
  = 3\bigl(
      119 x_1^4 + 580 x_1^3 x_2 + 1002 x_1^2 x_2^2
      + 768 x_1 x_2^3 + 240 x_2^4
    \bigr),
\]
a quartic bivariate polynomial which is not globally nonnegative, but is
nonnegative on the nonnegative orthant $\R^2_{\ge0}$. In particular, this
quartic is not a sum of squares, yet it satisfies a Rayleigh-type inequality
on $\R^2_{\ge0}$. This motivates our cone-restricted Rayleigh inequalities for
$\K$-Lorentzian polynomials.

\begin{theorem}\label{them:rayleighdifference}
Let $f \in \R_n^d[x]$ be a $\K$-Lorentzian polynomial over a proper convex cone
$\K \subset \R^n$. Then for every $x \in \K$ and every direction $u \in \R^n$,
the Rayleigh-type difference
\[
  R_u f(x)
  \;:=\;
  \bigl(D_u f(x)\bigr)^2 - f(x)\,D_u^2 f(x)
\]
is nonnegative. In particular, $R_u f(x) \ge 0$ on $\K$ for all $u \in \K$.
\end{theorem}

\begin{proof}
Set
\[
  M(x) \coloneqq \nabla f(x)\,\nabla f(x)^{\mathsf T} - f(x)\,H_f(x),
\]
where $H_f(x) := \nabla^{2} f(x)$ denote the Hessian of $f$ at $x$.
. Since $f$ is $\K$-Lorentzian, it is
$\K$-completely log-concave, so $-\nabla^2 \log f(x)$ is positive semidefinite
for all $x \in \inter \K$. Using
\[
  \nabla^2 \log f(x) = \frac{1}{f(x)} H_f(x)
                      - \frac{1}{f(x)^2} \nabla f(x)\,\nabla f(x)^{\mathsf T},
\]
we obtain
\[
  -\nabla^2 \log f(x)
  = \frac{1}{f(x)^2} M(x),
\]
so $M(x)$ is positive semidefinite for all $x \in \inter \K$, and by continuity
also for $x \in \K$.

For any $u \in \R^n$ we then have $u^{\mathsf T} M(x) u \ge 0$. A direct
computation shows
\[
  u^{\mathsf T} M(x) u
  = \bigl(\nabla f(x)^{\mathsf T} u\bigr)^2
    - f(x)\,u^{\mathsf T} H_f(x) u
  = \bigl(D_u f(x)\bigr)^2 - f(x)\,D_u^2 f(x)
  = R_u f(x),
\]
so $R_u f(x) \ge 0$ for all $x \in K$ and all $u \in \R^n$.
\end{proof}

Define the two-direction Rayleigh difference at $x$ by
\[
  R_{v,w} f(x)
  \coloneqq D_v f(x)\,D_w f(x) - f(x)\,D_v D_w f(x),
  \qquad v,w \in \R^n,
\]
and keep the notation
\[
  M(x) \coloneqq \nabla f(x)\,\nabla f(x)^{\mathsf T} - f(x)\,H_f(x),
\]
so that $R_{v,w} f(x) = v^{\mathsf T} M(x)\,w$.

\begin{proposition}\label{prop:two-dir-rayleigh}
Let $f \in \R_n^d[x]$ be $\K$-Lorentzian over a proper convex cone
$\K \subset \R^n$, so that $M(x)\succeq 0$ for all $x\in \inter \K$.
Fix $x\in \inter \K$. Then the following are equivalent:
\begin{enumerate}[(i)]
  \item $R_{v,w} f(x) \ge 0$ for all $v,w \in \K$.
  \item $v^{\mathsf T}M(x)\,w \ge 0$ for all $v,w \in \K$.
  \item The cone $\K$ is \emph{acute} with respect to the symmetric
  bilinear form
  \[
    \langle v,w\rangle_{M(x)} \coloneqq v^{\mathsf T} M(x)\,w,
  \]
  i.e., $\langle v,w\rangle_{M(x)} \ge 0$ for all $v,w \in \K$.
\end{enumerate}
If, in addition, $\K$ is finitely generated,
\[
  \K = \operatorname{cone}\{u_1,\dots,u_m\},
\]
then these are further equivalent to
\begin{enumerate}[(iv)]
  \item $u_i^{\mathsf T} M(x)\,u_j \ge 0$ for all $i,j \in \{1,\dots,m\}$.
\end{enumerate}
\end{proposition}

\begin{proof}
The equivalence (i) $\Leftrightarrow$ (ii) is just the identity
$R_{v,w} f(x) = v^{\mathsf T} M(x)\,w$. The equivalence (ii)
$\Leftrightarrow$ (iii) is a matter of terminology: a cone is acute with
respect to a symmetric bilinear form $\beta(\cdot,\cdot)$ if
$\beta(v,w)\ge 0$ for all $v,w$ in the cone, and here we take
$\beta(v,w) = v^{\mathsf T} M(x)\,w$.

Now assume $K = \operatorname{cone}\{u_1,\dots,u_m\}$. If (iv) holds and
$v,w\in K$ are written as
\[
  v = \sum_{i=1}^m \alpha_i u_i,\qquad
  w = \sum_{j=1}^m \beta_j u_j,
  \qquad \alpha_i,\beta_j \ge 0,
\]
then by bilinearity of $\langle\cdot,\cdot\rangle_{M(x)}$ we obtain
\[
  v^{\mathsf T}M(x)\,w
  = \sum_{i,j} \alpha_i \beta_j\,u_i^{\mathsf T}M(x)\,u_j
  \ge 0,
\]
since each term in the sum is nonnegative. Thus (iv) implies (ii).
Conversely, if (ii) holds, then in particular
$u_i^{\mathsf T} M(x)\,u_j \ge 0$ for all $i,j$, so (iv) holds. This
gives the equivalence (ii) $\Leftrightarrow$ (iv).
\end{proof}

\begin{corollary}\label{cor:two-dir-rayleigh}
Let $f$ be $\K$-Lorentzian and $M(x)$ as above. For each $x\in \inter \K$, the
diagonal Rayleigh inequality
\[
  (D_u f(x))^2 - f(x)\,D_u^2 f(x) = u^{\mathsf T}M(x)\,u \;\ge 0
\]
holds for all $u\in\R^n$. Moreover, the two-direction Rayleigh difference
\[
  R_{v,w} f(x)
  = D_v f(x)\,D_w f(x) - f(x)\,D_v D_w f(x)
\]
is nonnegative for all $v,w\in \K$ if and only if the cone $\K$ is acute with
respect to the bilinear form $\langle\cdot,\cdot\rangle_{M(x)}$, equivalently,
if $u_i^{\mathsf T}M(x)\,u_j \ge 0$ for all pairs of generating rays
$u_i,u_j$ of $\K$.
\end{corollary}

\begin{remark}
In the special case $\K = \R^n_{\ge 0}$ with standard basis $(e_1,\dots,e_n)$,
the entries $e_i^{\mathsf T} M_f(x)\,e_j$ are exactly the coordinate Rayleigh
differences
\[
  \Delta_{ij} f(x)
  \coloneqq
  \frac{\partial f}{\partial x_i}(x)\,
  \frac{\partial f}{\partial x_j}(x)
  - f(x)\,\frac{\partial^2 f}{\partial x_i \partial x_j}(x).
\]
Thus, if a multi-affine $f$ satisfies $\Delta_{ij} f(x)\ge0$ for all $x$ and
all $i,j$, then $\R^n_{\ge0} = \operatorname{cone}\{e_1,\dots,e_n\}$ is acute
with respect to $M_f(x)$, and by Proposition~\ref{prop:two-dir-rayleigh} we
obtain
\[
  R_{v,w} f(x) \;\ge\; 0
  \quad\text{for all }x\in\R^n \text{ and all } v,w\in\R^n_{\ge0}.
\]
For real stable multi-affine polynomials this recovers the classical global
Rayleigh inequalities $\Delta_{ij} f(x)\ge0$ and extends them from coordinate
directions to arbitrary nonnegative directions.
\end{remark}

\subsection*{Acuteness for quadratic forms vs.\ Rayleigh matrices}

We briefly distinguish two notions of acuteness that appear in the
$\K$-Lorentzian setting.

\begin{definition}\label{def:acute-Q}
Let $\K\subset\R^n$ be a closed convex cone and let $Q\in\R^{n\times n}$ be
symmetric. We say that $\K$ is \emph{acute with respect to $Q$} if
\[
  Q(\K) \subseteq \K^*
  \quad\Longleftrightarrow\quad
  y^{\mathsf T} Q x \ge 0 \quad\text{for all } x,y\in \K,
\]
where $\K^* := \{y\in\R^n : y^{\mathsf T}x \ge 0\ \forall x\in K\}$ is the dual
cone.
\end{definition}

For a quadratic form $q(x)=x^{\mathsf T}Qx$ the $\K$-Lorentzian condition is
equivalent to (\cite{GPlorentzian}):
\begin{enumerate}[(i)]
  \item $Q$ has exactly one positive eigenvalue, and
  \item $Q(\K)\subseteq \K^*$, i.e.\ $\K$ is acute with respect to $Q$.
\end{enumerate}
Thus, for every quadratic form obtained as a repeated directional derivative
\[
  q(x) = D_{a_1}\cdots D_{a_{d-2}} f(x)
\]
with $a_1,\dots,a_{d-2}\in\operatorname{int}K$, the associated matrix $Q$
satisfies $y^{\mathsf T}Qx\ge 0$ for all $x,y\in \K$.

A matrix $Q$ is (strictly) $\K$-copositive if $x^{t}Qx (>0) \geq 0$ for all $x \in \K$.
\begin{corollary} \label{cor:copositive}
If a quadratic form $q(x)=x^{t}Qx\in \R[x]_n^2$ is (strictly) $\K$-Lorentzian, then $Q$ is (strictly) $\K$-copositive matrix.
\end{corollary}

\begin{remark}\label{rem:acute-distinction}
It is important to note that these two notions of acuteness are a priori
different. By definition of $\K$-Lorentzian quadratics, $\K$ is always acute with
respect to the constant matrices $Q$ arising from quadratic directional
derivatives of $f$. In contrast, the matrices $M_f(x)$ depend on the point
$x\in \K$, and the $\K$-Lorentzian property of $f$ only guarantees
$M_f(x)\succeq 0$, which implies the diagonal Rayleigh inequality
$R_u f(x)\ge 0$ for all $u$, but does \emph{not} automatically imply
$v^{\mathsf T}M_f(x)\,w\ge 0$ for all $v,w\in \K$. Thus nonnegativity of the
two-direction Rayleigh differences $R_{v,w} f(x)$ on $\K$ requires, in addition
to $\K$-Lorentzianity, the extra geometric condition that $\K$ be acute with
respect to $M_f(x)$ for all $x\in\operatorname{int}\K$.
\end{remark}

\begin{definition}[{$\K$-Rayleigh measure}] \label{def:rayleigh_measure}
Let $\K\subset\R^m$ be a convex cone. A probability measure $\mu$ on
$2^{[m]}$ is called \emph{$\K$-Rayleigh} if its partition function $Z$ satisfies
\[
  \Delta_{ij} Z(w) \;\ge 0
  \quad\text{for all }w\in \K\text{ and all }i,j\text{ with }e_i,e_j\in \K.
\]
Equivalently, for every $w\in\operatorname{int} \K$ and such $i,j$,
the tilted measure $\mu_w$ satisfies
\begin{equation*}
  \PP_{\mu_w}(i,j\in S)
  \;\le\;
  \PP_{\mu_w}(i\in S)\PP_{\mu_w}(j\in S).
\end{equation*}
\end{definition}
From a probabilistic viewpoint, the Rayleigh matrix
\[
  M_f(x)
  \coloneqq \nabla f(x)\,\nabla f(x)^{\mathsf T} - f(x)\,H_f(x)
\]
is a local negative covariance/curvature operator for the Gibbs measure
$\mu_x(\alpha)\propto c_\alpha x^\alpha$ associated with
$f(x)=\sum_\alpha c_\alpha x^\alpha$ with nonnegative coefficients. Since
$\log f(x)$ is the log-partition function, our earlier identities yield
\[
  M_f(x) = -\,f(x)^2\,\nabla^2 \log f(x),
\]
so spectral lower bounds on $M_f(x)$ are equivalent to strong log-concavity of
$\mu_x$, which in turn implies Poincaré and log--Sobolev inequalities,
concentration, and fast mixing for natural Markov chains via the
Bakry-Émery $\Gamma_2$-calculus and related methods
\cite{BrascampLieb,BakryEmery,AnariEtAlLC3}. In contrast, the classical scalar
Rayleigh differences
\[
  \Delta_{ij} f(x)
  = \frac{\partial f}{\partial x_i}(x)\,\frac{\partial f}{\partial x_j}(x)
    - f(x)\,\frac{\partial^2 f}{\partial x_i \partial x_j}(x)
\]
control coordinatewise negative dependence and underlie strongly Rayleigh and
negatively associated measures \cite{BorceaBrandenLiggett}. The matrix
$M_f(x)$ is a matrix-valued refinement: for any directions $v,w$,
\[
  R_{v,w} f(x)
  = D_v f(x)\,D_w f(x) - f(x)\,D_v D_w f(x)
  = v^{\mathsf T}M_f(x)\,w,
\]
so while $\Delta_{ij} f$ probes pairwise dependence along coordinates, the full
Rayleigh matrix encodes the second-order geometry of $\log f$ and is
better suited for spectral and mixing questions in log-concave sampling and
spectral independence
\cite{AnariSpectralIndependence,Cynthialog,WainwrightJordan}.

For completeness we record the standard calculus identities relating derivatives
of the log-partition function $\log f(x)$ to the moments of the associated
Gibbs measure; see, for example,
\cite{WainwrightJordan,BorceaBrandenLiggett,Cynthialog}. Let
\[
  f(x) \;=\; \sum_{\alpha \in \mathbb{N}^n} c_\alpha\,x^\alpha,
  \qquad c_\alpha \ge 0,\quad x \in \mathbb{R}^n_{>0},
\]
and define the Gibbs measure
\[
  \mu_x(\alpha)
  \;\coloneqq\;
  \frac{c_\alpha\,x^\alpha}{f(x)}.
\]
For each coordinate $i$, we have
\[
  \frac{\partial f}{\partial x_i}(x)
  \;=\;
  \sum_\alpha c_\alpha\,\alpha_i\,x^{\alpha-e_i},
  \qquad
  x_i\,\frac{\partial f}{\partial x_i}(x)
  \;=\;
  \sum_\alpha c_\alpha\,\alpha_i\,x^{\alpha},
\]
and hence
\[
  \frac{\partial}{\partial x_i} \log f(x)
  \;=\;
  \frac{1}{f(x)}\,\frac{\partial f}{\partial x_i}(x)
  \;=\;
  \frac{1}{x_i}\,\sum_\alpha \mu_x(\alpha)\,\alpha_i
  \;=\;
  \frac{\mathbb{E}_{\mu_x}[\alpha_i]}{x_i}.
\]
Thus the gradient of $\log f$ encodes the (scaled) mean parameters of $\mu_x$.
Differentiating once more, we obtain
\[
  \frac{\partial^2}{\partial x_i\,\partial x_j} \log f(x)
  \;=\;
  \frac{1}{f(x)}\,\frac{\partial^2 f}{\partial x_i\,\partial x_j}(x)
  \;-\;
  \frac{1}{f(x)^2}\,
  \frac{\partial f}{\partial x_i}(x)\,
  \frac{\partial f}{\partial x_j}(x).
\]
A direct computation shows
\[
  x_i x_j\,\frac{\partial^2 f}{\partial x_i\,\partial x_j}(x)
  \;=\;
  \sum_\alpha c_\alpha\,\alpha_i(\alpha_j-\delta_{ij})\,x^\alpha,
\]
so
\[
  \frac{1}{f(x)}\,\frac{\partial^2 f}{\partial x_i\,\partial x_j}(x)
  \;=\;
  \frac{1}{x_i x_j}\,
  \mathbb{E}_{\mu_x}\!\bigl[\alpha_i(\alpha_j-\delta_{ij})\bigr],
\]
and, using the formula for $\partial f/\partial x_i$ above,
\[
  \frac{1}{f(x)^2}\,
  \frac{\partial f}{\partial x_i}(x)\,
  \frac{\partial f}{\partial x_j}(x)
  \;=\;
  \frac{1}{x_i x_j}\,
  \mathbb{E}_{\mu_x}[\alpha_i]\,
  \mathbb{E}_{\mu_x}[\alpha_j].
\]
Hence, for $i \neq j$,
\[
  \frac{\partial^2}{\partial x_i\,\partial x_j} \log f(x)
  \;=\;
  \frac{1}{x_i x_j}\,
  \bigl(
    \mathbb{E}_{\mu_x}[\alpha_i \alpha_j]
    - \mathbb{E}_{\mu_x}[\alpha_i]\,
      \mathbb{E}_{\mu_x}[\alpha_j]
  \bigr)
  \;=\;
  \frac{1}{x_i x_j}\,
  \operatorname{Cov}_{\mu_x}(\alpha_i,\alpha_j),
\]
and for $i=j$ one gets a similar expression in terms of
$\operatorname{Var}_{\mu_x}(\alpha_i)$ and $\mathbb{E}_{\mu_x}[\alpha_i]$. In
particular, in logarithmic coordinates $\theta_i = \log x_i$ one has
(cf.\ the standard exponential-family identities \cite{WainwrightJordan})
\[
  \frac{\partial}{\partial \theta_i} \log f(e^\theta)
  \;=\;
  \mathbb{E}_{\mu_{e^\theta}}[\alpha_i],
  \qquad
  \frac{\partial^2}{\partial \theta_i\,\partial \theta_j} \log f(e^\theta)
  \;=\;
  \operatorname{Cov}_{\mu_{e^\theta}}(\alpha_i,\alpha_j),
\]
so the Hessian of the log-partition function in $\theta$-coordinates is exactly
the covariance matrix of $\mu_{e^\theta}$. Combining this with
$M_f(x) = -f(x)^2 \nabla^2\log f(x)$ shows that $M_f(x)$ is, up to the positive
factor $f(x)^2$ and the change of variables $x_i = e^{\theta_i}$, a negative
covariance operator for the family of Gibbs measures
$(\mu_x)_{x\in\mathbb{R}^n_{>0}}$.

\section{Semipositive cones and hyperbolic barriers in conic optimization}
\label{sec:semipositive}
In this section we relate nonsingular matrices and their semipositive cones to hyperbolic generating polynomials that serve as natural barrier functions for conic optimization problems.
\subsection{Hyperbolic generating polynomials}
We first explain how to construct hyperbolic generating polynomials from a
given nonsingular matrix $A\in\R^{n\times n}$.

\begin{proposition}\label{prop:hyp}
If $A=(a_{ij})$ is nonsingular, then the generating polynomial
\[
  f_A(x)
  \;=\;
  \det\Bigl(\sum_{j=1}^n x_j D_j\Bigr),
  \qquad
  D_j = \diag(a_{1j},\dots,a_{nj}),
\]
is hyperbolic with respect to some direction $e\in\R^n$.
\end{proposition}

\begin{proof}
We have
\[
  \sum_{j=1}^n x_j D_j
  =
  \diag\Bigl(\sum_{j=1}^n a_{1j}x_j,\dots,\sum_{j=1}^n a_{nj}x_j\Bigr),
\]
so $f_A$ is a determinantal polynomial in the diagonal matrices $D_j$. By
\cite{Pablolax}, such a determinantal polynomial is hyperbolic if the linear
span of the $D_j$ contains a positive definite matrix. Without loss of
generality we may take this matrix to be the identity $I$, i.e.\ we seek
$e=(e_1,\dots,e_n)\in\R^n$ such that
\[
  \sum_{j=1}^n e_j D_j = I.
\]
This is equivalent to $Ae=\mathbf{1}$, where $\mathbf{1}$ is the all-ones
vector. Since $A$ is nonsingular, the system $Ae=\mathbf{1}$ has a unique
solution $e\in\R^n$, and $\sum_j e_j D_j = I\succ0$, so $f_A$ is hyperbolic
with respect to this $e$.
\end{proof}

\noindent\textbf{Special cases.}
\begin{enumerate}
  \item If $A$ is nonsingular with nonnegative entries, then $f_A$ is stable
  (each diagonal entry $\sum_j a_{ij}x_j$ has nonnegative coefficients).
  Doubly stochastic matrices are a basic example.
  \item If $A$ is symmetric positive definite, then $f_A$ is hyperbolic and
  the monomial $x_1\cdots x_n$ appears with positive coefficient.
\end{enumerate}

\begin{example}\label{example:1}
Consider
\[
  A =
  \begin{bmatrix}
    -1 & 1 & 1 & 1 \\
    1 & -1 & 1 & 1 \\
    1 & 1 & -1 & 1 \\
    1 & 1 & 1 & -1
  \end{bmatrix}.
\]
The generating polynomial
\[
  f_A(x)
  = \det\Bigl(\sum_{j=1}^4 x_j D_j\Bigr),
  \qquad
  D_j = \diag(a_{1j},\dots,a_{4j}),
\]
is
\[
  f_A(x)
  = -\sum_{i=1}^4 x_i^4
    + 2\sum_{1\le i<j\le 4} x_i^2 x_j^2
    + 8 x_1 x_2 x_3 x_4.
\]
A direct computation shows that the Hessian at $\mathbf{1}=(1,1,1,1)$ is
\[
  H_{f_A}(\mathbf{1})
  = 16
  \begin{bmatrix}
    0 & 1 & 1 & 1 \\
    1 & 0 & 1 & 1 \\
    1 & 1 & 0 & 1 \\
    1 & 1 & 1 & 0
  \end{bmatrix}
  =: 16 B,
\]
where $B$ is the adjacency matrix of the complete graph $K_4$. This is not a
coincidence: $f_A$ is invariant under all permutations of the coordinates, so
$H_{f_A}(\mathbf{1})$ must be invariant under the action of the symmetric
group $S_4$, hence of the form $\alpha I + \beta(J-I)$ with $J$ the all-ones
matrix. Since the diagonal entries vanish, $\alpha=0$ and
$H_{f_A}(\mathbf{1})$ is a constant multiple of $J-I$, i.e.\ of the adjacency
matrix of $K_4$. The eigenvalues of $B$ are $3,-1,-1,-1$, so
$H_{f_A}(\mathbf{1})$ has exactly one positive eigenvalue, and $f_A$ is
$\K$-Lorentzian for some cone $\K$ containing the ray through $\mathbf{1}$.
By \Cref{prop:hyp}, $f_A$ is hyperbolic with respect to a suitable direction
$e$.
\end{example}

\begin{example}[Motivation from conic optimization]\label{ex:practical-semipositive}
A basic setting where hyperbolic generating polynomials arise is the conic
optimization problem
\[
  \min\{\, c^{\mathsf T}x : x \in \R^n_{\ge0},\ Ax \ge 0 \,\}.
\]
Here $x$ collects nonnegative activity levels and $Ax\ge0$ encodes linear
balance or safety constraints at $n$ subsystems. The feasible region is the
\emph{semipositive cone}
\[
  K_A = \{x\in\R^n_{\ge0} : Ax\ge0\}.
\]
If $A$ is nonsingular, then by \Cref{prop:hyp} the generating polynomial
$f_A(x) = \det(\sum_{j=1}^n x_j D_j)$ is hyperbolic, and
\[
  \Lambda_+(f_A,e)
  = \{x\in\R^n : Ax\ge0\},
  \qquad
  K_A = \Lambda_+(f_A,e)\cap\R^n_{\ge0}.
\]
Thus $K_A$ is precisely the part of the hyperbolicity cone that lies in the
positive orthant, and $f_A$ provides a natural hyperbolic barrier on
$\operatorname{int}K_A$. Characterizing when $\Lambda_+(f_A,e)$ intersects
$\R^n_{>0}$ (equivalently, when $A$ is semipositive) tells us exactly when we
have a strictly feasible region with a hyperbolic barrier compatible with the
ambient modeling cone $\R^n_{\ge0}$. This motivates our study of semipositive
cones and their $\K$-Lorentzian generating polynomials.
\end{example}
The previous examples suggest that the intersection
$\Lambda_+(f_A,e)\cap\R^n_{\ge0}$ plays a distinguished role. We now formalize
this by introducing semipositive cones and their generalization to arbitrary
proper convex cones. We show that when $A$ is nonsingular and
semipositive, the hyperbolicity cone of its generating polynomial meets
$\R^n_{\ge0}$ in the \emph{semipositive cone} of $A$.

A matrix $A\in\R^{n\times n}$ is called \emph{semipositive} if there exists
$x>0$ such that $Ax>0$. The associated semipositive cone is
\[
  \operatorname{int}K_A
  \;:=\;
  \{x\in\R^n_{>0} : Ax>0\},
  \qquad
  K_A
  \;:=\;
  \{x\in\R^n_{\ge0} : Ax\ge0\}.
\]
It is known that if $A$ is semipositive, then $K_A$ is a proper convex cone in
$\R^n$; moreover, $K_A$ is polyhedral and can be written as the intersection
of the nonnegative orthant with the cone generated by the columns of $A$, see
\cite{geometricmapping}.

\begin{proposition}\label{prop:simhypcone}
Let $A\in\R^{n\times n}$ be nonsingular and semipositive, and let $f_A$ be its
generating polynomial (as in \Cref{prop:hyp}). Then
\[
  \Lambda_+(f_A,e)\cap\R^n_{\ge0} = K_A,
\]
i.e.\ the intersection of the hyperbolicity cone of $f_A$ with the nonnegative
orthant is exactly the semipositive cone $K_A$.
\end{proposition}

\begin{proof}
By \Cref{prop:hyp} the generating polynomial $f_A$ is hyperbolic with respect
to some $e$. Moreover,
\[
  f_A(x) = \det\Bigl(\sum_{j=1}^n x_j D_j\Bigr),
\]
where each $D_j$ is the diagonal matrix whose diagonal entries are the entries
of the $j$th column of $A$. Hence $f_A$ is a determinantal polynomial and its
hyperbolicity cone is
\[
  \Lambda_+(f_A,e)
  = \Bigl\{x\in\R^n : \sum_{j=1}^n x_j D_j \succeq 0\Bigr\}
  = \{x\in\R^n : Ax\ge0\},
\]
since $\sum_j x_j D_j$ is the diagonal matrix with diagonal entries $(Ax)_1,
\dots,(Ax)_n$. Therefore
\[
  \Lambda_+(f_A,e)\cap\R^n_{\ge0}
  = \{x\in\R^n_{\ge0} : Ax\ge0\}
  = K_A,
\]
a proper polyhedral cone. \qedhere
\end{proof}

\begin{corollary}\label{cor:simhypcone}
Let $A$ be a nonsingular matrix. Then $f_A$ is a $\K$-Lorentzian (and in
particular hyperbolic) polynomial with respect to
$\K=\Lambda_+(f_A,e)$, a simplicial cone.
\end{corollary}

\begin{example}
Consider the matrix $A$ from \Cref{example:1}. Then the corresponding
semipositive cone $K_A=\{x\in\R^n_{\ge0}:Ax\ge0\}$ is depicted in
\Cref{fig:semipositivecone}.

\begin{figure}[htbp]
  \begin{center}
    \includegraphics[width=5cm]{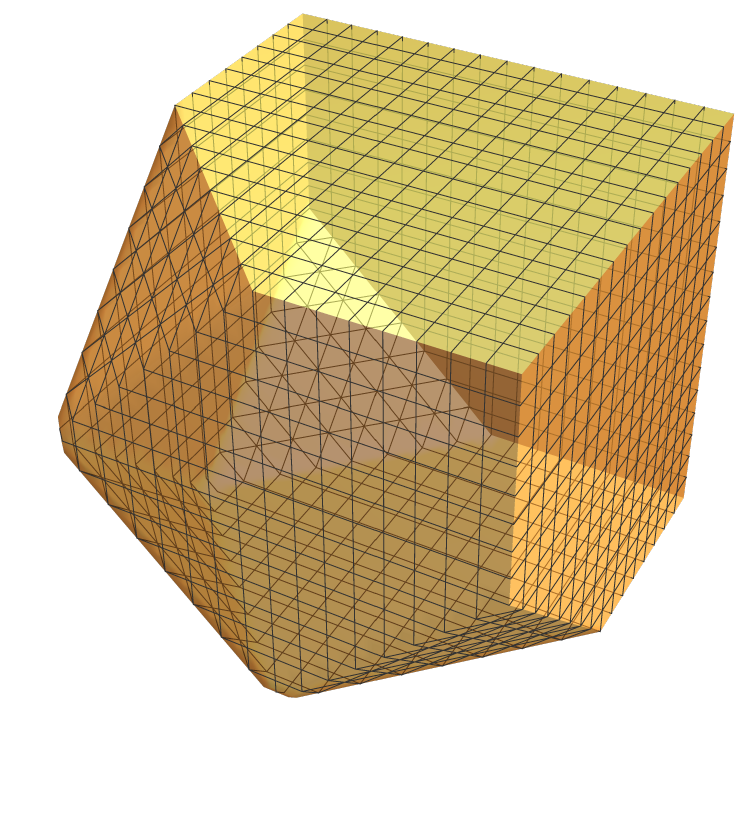}
    \caption{Compact base of $\K_A=\Lambda_{+}(f_A,1)\cap\R^4_{\ge0}$ at $x_4=1$.}
    \label{fig:semipositivecone}
  \end{center}
\end{figure}
\end{example}

\subsection*{Generalization: \texorpdfstring{$\K$}{K}--semipositive cones}

We now generalize semipositive matrices and semipositive cones from the
nonnegative orthant to an arbitrary proper convex cone $\K$.

Let $\K\subset\R^n$ be a proper convex cone. A matrix $A\in\R^{n\times n}$ is
called \emph{$\K$-semipositive} if there exists $x\in\operatorname{int}\K$
such that $Ax\in\operatorname{int}\K$, i.e.
\[
  A(\operatorname{int}\K)\cap\operatorname{int}\K \neq \emptyset.
\]
We say $A$ is \emph{$\K$-nonnegative} if $A(\K)\subseteq\K$, and
\emph{$\K$-irreducible} in the usual sense of cone theory.

\begin{theorem}\label{them:simcone}
\cite[Chap.~5, Th.~5.1]{Bermannonnegative}
Let $\K$ be a proper convex cone.
\begin{enumerate}[(a)]
\item If $A$ is $\K$-irreducible and $A(\K) = \K$, then $A^{-1}$ is also
$\K$-irreducible and $A(\partial \K) = A^{-1}(\partial \K) = \partial \K$.
\item Let $A$ be $\K$-nonnegative and $\K$-semipositive. Then $A(\K) = \K$.
\end{enumerate}
\end{theorem}

As a basic example, the matrix $J=\diag(1,-1,\dots,-1)$ is
$\L_n$--irreducible and satisfies $J(\L_n)=\L_n$, where $\L_n$ is the
$n$-dimensional second-order (Lorentz) cone. More generally, matrices that
preserve a cone and are $\K$-irreducible play an important role in
optimization when the feasible region is the intersection of two proper convex
cones.

\begin{remark}
If $A(\K)=\K$ and $A$ is nonsingular, then both $A$ and $A^{-1}$ are
$\K$-nonnegative.
\end{remark}

\begin{remark}
If $A$ is nonsingular and $\K$-irreducible with $A(\K)=\K$, then the
eigenvalues of $A$ all have the same modulus (Perron--Frobenius for
cone-preserving maps). The matrix $A$ in \Cref{example:1} satisfies these
conditions over its hyperbolicity cone.
\end{remark}

In \Cref{example:1}, the matrix $A$ is both $\K$-irreducible and
$\R^4_{\ge0}$--semipositive. Therefore $f_A$ is $K_A$--Lorentzian on the
semipositive cone
\[
  K_A
  = \{x\in\R^n_{\ge0}:Ax\ge0\}
  = \Lambda_+(f_A,e)\cap\R^n_{\ge0},
\]
a proper polyhedral cone. Combining \Cref{cor:simhypcone} and
\Cref{them:simcone} we obtain the following generalization.

\begin{theorem}\label{them:clcsemipositive}
Let $\tilde{\K}$ be a proper convex cone containing $e$, and let $A$ be
nonsingular and $\tilde{\K}$--semipositive. Then the generating polynomial
$f_A$ is $\K_A$--Lorentzian, where
\[
  K_A
  \;:=\;
  \Lambda_+(f_A,e)\cap\tilde{\K}
  \;=\;
  \{x\in\tilde{\K}:Ax\ge0\}
\]
is a $\tilde{\K}$--semipositive cone.
\end{theorem}

\begin{proof}
Since $A$ is nonsingular, \Cref{cor:simhypcone} implies that $f_A$ is
$\Lambda_+(f_A,e)$--Lorentzian for some $e\in\Lambda_+(f_A,e)$, and
\Cref{prop:simhypcone} gives
\[
  \Lambda_+(f_A,e)=\{x\in\R^n:Ax\ge0\}.
\]
By $\tilde{\K}$--semipositivity and $e\in\tilde{\K}$, the intersection
\[
  K_A
  = \Lambda_+(f_A,e)\cap\tilde{\K}
  = \{x\in\tilde{\K}:Ax\ge0\}
\]
is a nonempty $\tilde{\K}$--semipositive cone. Since $K_A\subseteq
\Lambda_+(f_A,e)$, restricting all directional derivatives to $K_A$ shows that
$f_A$ is $K_A$--Lorentzian.
\end{proof}

It is shown in \cite[Sec.~8]{Gulerhyperbolicconvex} that every homogeneous
convex cone admits a hyperbolic barrier function. While not every hyperbolicity
cone is homogeneous, the preceding discussion illustrates a natural class of
examples where the hyperbolicity cone interacts well with cone-preserving
linear maps and semipositive cones.

\begin{example}
For the matrix $A$ in \Cref{example:1} one has
$A\bigl(\Lambda_+(f_A,1)\bigr)=\Lambda_+(f_A,1)$.
\end{example}

\begin{remark}
Semipositive cones provide a natural interface between hyperbolic geometry and
the classical theory of cone-preserving linear maps. On the one hand, for a
nonsingular semipositive matrix $A$, the cone
\[
  K_A = \Lambda_+(f_A,e)\cap\R^n_{\ge0}
\]
identifies precisely the part of the hyperbolicity cone that lives in the
positive orthant. On the other hand, $K_A$ is also the invariant cone of the
cone-preserving map $x\mapsto Ax$, so tools from Perron--Frobenius theory
($K$--irreducibility, spectral radius, invariant faces) can be brought to bear.
From the viewpoint of optimization, such semipositive cones naturally arise as
feasible regions given by the intersection of two proper convex cones (for
instance, a hyperbolicity cone and a standard modeling cone such as
$\R^n_{\ge0}$ or a second-order cone). Thus semipositive cones serve as a
bridge between the algebraic structure of determinantal hyperbolic polynomials
and the geometric and spectral properties of cone-preserving linear dynamics,
which is particularly useful when designing and analysing conic optimization
problems with hyperbolic barriers.
\end{remark}

\section{Stability Analysis via \texorpdfstring{$\K$}{}-Lorentzian polynomials} \label{sec:dynamic}
In this section we study \emph{cone-stability} of evolution variational
inequality (EVI) systems: stability of an equilibrium when the state is
constrained to lie in a closed convex cone $\K$. Even if an equilibrium is
unstable in the full space, it may become stable once trajectories are
confined to $\K$. Motivated by such examples, we focus on EVI systems as in
\cite{stabilitylevi,Henrionconecopositive} and recall the basic framework
needed to introduce $\K$-Lorentzian Lyapunov functions.

The notation $C^{0}([t_0,+\infty);\R^{d})$ denotes continuous functions
$[t_0,+\infty)\to\R^{d}$, and $L^{\infty}_{\mathrm{loc}}(t_0,+\infty;\R^{d})$
denotes locally essentially bounded functions $(t_0,+\infty)\to\R^{d}$. We
write $C^{1}(\R^{d};\R)$ for continuously differentiable scalar-valued
functions on $\R^{d}$.

\medskip\noindent
\textit{Class of dynamical systems.}
Let $\K\subset\R^{d}$ be a nonempty closed convex set, $A\in\R^{d\times d}$ a
matrix, and $F:\R^{d}\to\R^{d}$ a nonlinear operator. For
$(t_0,x_0)\in\R\times\K$, consider the problem $P(t_0,x_0)$: find
$x:[t_0,+\infty)\to\R^{d}$ such that
$x\in C^{0}([t_0,+\infty);\R^{d})$, $\dot x\in
L^{\infty}_{\mathrm{loc}}(t_0,+\infty;\R^{d})$ and
\begin{equation}
\begin{cases}
\langle \dot x(t)+Ax(t)+F(x(t)),\,v-x(t)\rangle \ge 0,
&\forall v\in\K,\ \text{a.e. }t\ge t_0,\\[2pt]
x(t)\in\K,& t\ge t_0,\\[2pt]
x(t_0)=x_0.
\end{cases}
\label{eq:ds}
\end{equation}
This is an \emph{evolution variational inequality} (EVI) system.

For a closed convex set $\K$, the normal cone at $x\in\K$ is
\[
  N_{\K}(x)
  \;:=\;
  \{y\in\R^{d}:\ \langle y,v-x\rangle \le 0\ \forall v\in\K\},
\]
and the tangent cone is its polar
\[
  T_{\K}(x)
  \;:=\;
  \{d\in\R^{d}:\ \langle d,y\rangle \le 0\ \forall y\in N_{\K}(x)\}
  =
  \mathrm{cl}\{\alpha(v-x): v\in\K,\ \alpha\ge 0\}.
\]
Standard convex analysis gives the equivalent form of \eqref{eq:ds}:
\[
\begin{cases}
\dot x(t)+Ax(t)+F(x(t)) \in -N_{\K}(x(t)),\\[2pt]
x(t)\in\K,\ t\ge t_0.
\end{cases}
\]

We recall existence and uniqueness for $P(t_0,x_0)$, a special case of
\cite[Theorem~2.1 (Kato), Cor.~2.2]{Goelevenstability}.

\begin{theorem}[\cite{stabilitylevi}]\label{them:exisuniqivp}
Let $\K\subset\R^{d}$ be nonempty, closed, convex, and let
$A\in\R^{d\times d}$. Suppose $F:\R^{d}\to\R^{d}$ admits a decomposition
\[
  F = F_1 + \Phi',
\]
where $F_1$ is Lipschitz continuous and $\Phi\in C^{1}(\R^{d};\R)$ is convex.
For given $t_0\in\R$ and $x_0\in\K$ there exists a unique solution
$x(\cdot;t_0,x_0)\in C^{0}([t_0,+\infty);\R^{d})$ with
$\dot x\in L^{\infty}_{\mathrm{loc}}(t_0,+\infty;\R^{d})$, right
differentiable on $[t_0,+\infty)$, satisfying \eqref{eq:ds}.
\end{theorem}

\begin{remark}
If in addition $0\in\K$ and $\langle F(0),v\rangle\ge 0$ for all $v\in\K$,
then $x(t;t_0,0)\equiv 0$ is the unique solution of $P(t_0,0)$. In what follows
we consider systems $P(t_0,x_0)$ satisfying the hypotheses of
\Cref{them:exisuniqivp}, together with $0\in\K$ and
$\langle F(0),v\rangle\ge0$ for all $v\in\K$, unless stated otherwise.
\end{remark}

We now recall stability notions for the equilibrium $x=0$ with respect to the
constraint set $\K$ \cite{stabilitylevi,Henrionconecopositive}.

\begin{definition}\label{def:constable}
The equilibrium $x=0\in\K$ is (Lyapunov) \emph{stable w.r.t.\ $\K$} if for any
$\epsilon>0$ there exists $\delta>0$ such that for any $x_0\in\K$ with
$\|x_0\|\le\delta$, the solution $x(t)\in\K$ of \eqref{eq:ds} satisfies
$\|x(t)\|<\epsilon$ for all $t\ge t_0$. It is \emph{asymptotically stable
w.r.t.\ $\K$} if it is stable and there exists $\delta>0$ such that
$\|x(t)\|\to0$ as $t\to\infty$ for all $x_0\in\K$ with $\|x_0\|\le\delta$.
\end{definition}

Throughout the paper we assume $\K$ is a proper convex cone; the results extend
to general closed convex sets $\K$ containing the origin
\cite{stabilitylevi}. We next recall abstract Lyapunov criteria for stability,
asymptotic stability, and instability w.r.t.\ $\K$ in terms of generalized
Lyapunov functions, cf.\ \cite{stabilitylevi}.

\begin{theorem}\label{them:dsstability}
Consider $P(t_0,x_0)$. Suppose there exist $\sigma>0$ and
$V\in C^{1}(\R^d;\R)$ such that
\begin{enumerate}
  \item $V(x)\ge a(\|x\|)$ for $x\in\K$, $\|x\|\le\sigma$, where
  $a:[0,\sigma]\to\R$ satisfies $a(t)>0$ for all $t\in(0,\sigma)$;
  \item $V(0)=0$;
  \item $x-\nabla V(x)\in\K$ for all $x\in\partial\K$ with $\|x\|\le\sigma$;
  \item $\langle Ax+F(x),\nabla V(x)\rangle \ge 0$ for all
  $x\in\K$, $\|x\|\le\sigma$.
\end{enumerate}
Then the trivial solution of $P(t_0,x_0)$ is stable w.r.t.\ $\K$.
\end{theorem}

\begin{theorem}\label{them:dsasymstability}
Consider $P(t_0,x_0)$. Suppose there exist $\lambda>0$, $\sigma>0$ and
$V\in C^{1}(\R^d;\R)$ such that
\begin{enumerate}
  \item $V(x)\ge a(\|x\|)$ for $x\in\K$, $\|x\|\le\sigma$, where
  $a:[0,\sigma]\to\R$ satisfies $a(t)\ge c t^\tau$ for all
  $t\in[0,\sigma]$, for some $c>0$, $\tau>0$;
  \item $V(0)=0$;
  \item $x-\nabla V(x)\in\K$ for all $x\in\partial\K$ with $\|x\|\le\sigma$;
  \item $\langle Ax+F(x),\nabla V(x)\rangle \ge \lambda V(x)$ for all
  $x\in\K$, $\|x\|\le\sigma$.
\end{enumerate}
Then the trivial solution of $P(t_0,x_0)$ is asymptotically stable w.r.t.\ $\K$.
\end{theorem}

\begin{remark}
Condition (3) implies $-\nabla V(x)\in T_{\K}(x)$ for all
$x\in\partial\K$, $\|x\|\le\sigma$, i.e.\ the negative gradient points
tangentially into $\K$.
\end{remark}

We now recall Lyapunov stability of matrices on $\K$ \cite{stabilitylevi}.

\begin{definition}\label{def:Lyapunovstable}
A matrix $A\in\R^{n\times n}$ is \emph{Lyapunov semi-stable} (resp.\ \emph{Lyapunov
positive stable}) on $\K$ if there exists $P\in\R^{n\times n}$ such that
\begin{enumerate}[(a)]
  \item $\displaystyle\inf_{x\in\K\setminus\{0\}}
          \frac{x^{\mathsf T}Px}{\|x\|^2} > 0$;
  \item $\langle Ax,[P+P^{\mathsf T}]x\rangle \ge 0$ (resp.\ $>0$)
        for all $x\in\K$;
  \item $x\in\partial\K \implies (I-[P+P^{\mathsf T}])x\in\K$.
\end{enumerate}
\end{definition}

Let $\mathcal{P}_{\K}$ (resp.\ $\mathcal{P}^{+}_{\K}$) denote the cone of
$\K$–copositive (resp.\ strictly $\K$–copositive) matrices and
\[
  \mathcal{P}^{++}_{\K}
  :=
  \Bigl\{P\in\R^{n\times n}:
    \inf_{x\in\K\setminus\{0\}} \frac{x^{\mathsf T}Px}{\|x\|^{2}} > 0
  \Bigr\}.
\]
Define
\[
  \mathcal{L}_{\K}
  :=
  \{A\in\R^{n\times n}:\ \exists P\in\mathcal{P}^{++}_{\K}
    \text{ such that } (I-[P+P^{\mathsf T}])\partial\K\subset\K,\ 
    A^{\mathsf T}P+PA\in\mathcal{P}_{\K}\}
\]
and
\[
  \mathcal{L}^{++}_{\K}
  :=
  \{A\in\R^{n\times n}:\ \exists P\in\mathcal{P}^{++}_{\K}
    \text{ such that } (I-[P+P^{\mathsf T}])\partial\K\subset\K,\ 
    A^{\mathsf T}P+PA\in\mathcal{P}^{+}_{\K}\}.
\]

\begin{remark}
The condition $\langle Ax,[P+P^{\mathsf T}]x\rangle > 0$ for all $x\in\K$
is equivalent to $A^{\mathsf T}P+PA\in\mathcal{P}^{+}_{\K}$. If $A$ is (classically)
strictly stable (all eigenvalues with positive real part in the convention of
\cite{stabilitylevi}), then by the Lyapunov theorem there exists $P\succ 0$
with $A^{\mathsf T}P+PA=Q\succ 0$, so (a) and (b) above hold automatically.
\end{remark}

\begin{lemma}[\cite{Goelevenstability,stabilitylevi}]
If $\K$ is a proper convex cone, then
$\mathcal{P}^{++}_{\K} = \mathcal{P}^{+}_{\K}$.
\end{lemma}

In the special case $F\equiv 0$, problem $P(t_0,x_0)$ reduces to the linear
EVI (LEVI) system
\begin{equation}
\begin{cases}
\langle \dot x(t)+Ax(t), v-x(t)\rangle \ge 0, &\forall v\in\K,\ \text{a.e. }t\ge t_0,\\[2pt]
x(t)\in\K, & t\ge t_0,\\[2pt]
x(t_0)=x_0.
\end{cases}
\label{eq:levi}
\end{equation}
Lyapunov stability of $A$ w.r.t.\ $\K$ then characterizes stability of the
trivial solution, see \cite[Theorem~5]{stabilitylevi}.

\begin{theorem}[\cite{stabilitylevi}]\label{them:copoconestable}
\begin{enumerate}[(a)]
  \item If $A$ is $\K$–copositive, then $A$ is Lyapunov semi-stable on $\K$.
  \item If $A$ is strictly $\K$–copositive, then $A$ is Lyapunov positive
  stable on $\K$.
\end{enumerate}
\end{theorem}

\begin{proof}
(a) Take $P=\tfrac{1}{2}I$. Conditions (a)–(b) in
\Cref{def:Lyapunovstable} are immediate, and
$(I-[P+P^{\mathsf T}])x=0\in\K$ for all $x$, so $A \in\mathcal{L}_{\K}$.
(b) The same choice $P=\tfrac{1}{2}I$ works when $A$ is strictly
$\K$–copositive.
\end{proof}

\begin{theorem}[\cite{stabilitylevi}]\label{them:LEVIstability}
Let $\K$ be a proper convex cone and consider the LEVI system
\eqref{eq:levi}. Then:
\begin{enumerate}[(a)]
  \item If $A\in\mathcal{L}_{\K}$, the trivial solution is stable w.r.t.\ $\K$.
  \item If $A\in\mathcal{L}^{++}_{\K}$, the trivial solution is
  asymptotically stable w.r.t.\ $\K$.
\end{enumerate}
\end{theorem}

\begin{proposition}\label{prop:clcconstable}
If the quadratic form $x^{\mathsf T}Ax$ is (strictly) $\K$-Lorentzian,
then the trivial solution of the LEVI
system \eqref{eq:levi} is (asymptotically) stable w.r.t.\ $\K$.
\end{proposition}

\begin{proof}
If $x^{\mathsf T}Ax$ is $\K$-Lorentzian (resp.\ strictly $\K$-Lorentzian),
then the symmetric unique matrix representation $A$ is same up to a factor of $2$ as the Hessian of the quadratic form. Hence $A$ is $\K$-copositive (resp.\ strictly $\K$-copositive) by
\Cref{cor:copositive}. By \Cref{them:copoconestable}, $A$ is then Lyapunov
semi-stable (resp.\ Lyapunov positive stable), and the claim follows from
\Cref{them:LEVIstability}.
\end{proof}

We briefly recall one example from \cite[Example~1]{stabilitylevi}. 
In $2$D, with
\[
  A =
  \begin{bmatrix}
    1 & 2\\
    1 & 1
  \end{bmatrix},
  \qquad
  \K = \R^2_{\ge0},
\]
one has $\mathrm{eig}(A)=\{1+\sqrt{2},\,1-\sqrt{2}\}$, so $0$ is unstable in
the unconstrained system, but becomes asymptotically stable for the LEVI system
on $\K$. 
We now illustrate \Cref{prop:clcconstable} and the $\K$–Lorentzian viewpoint with a three-dimensional example.

\begin{example}\label{ex:levi}
Consider the quadratic form
\[
f(x_1,x_2,x_3)
= x_1^2 + 4x_1x_2 + x_2^2 + 4x_1x_3 + 4x_2x_3 - x_3^2
= x^{\mathsf T} A x,
\]
where
\[
A=
\begin{bmatrix}
1 & 2 & 2\\
2 & 1 & 2\\
2 & 2 & -1
\end{bmatrix}.
\]
Since one coefficient of $x_3^2$ is negative, \(f\) is not Lorentzian on \(\R^3_{\ge 0}\). In particular, \(f\) is not positive on \(\R^3_{>0}\). Fix \(v=(1,1,1)^{\mathsf T}\). The matrix \(A\) is symmetric with signature \((+,-,-)\), and \(f(v)=v^{\mathsf T}Av=13>0\). Hence \(f\) is hyperbolic with respect to \(v\), see \cite{Deycomplete}. Therefore, the hyperbolicity cone \(K(f,v)\) (the connected component of \(\{x\in\R^3:f(x)>0\}\) containing \(v\)) is a convex cone. Moreover, \(\{x:f(x)>0\}\) has exactly two connected components, each of which is a convex cone.

Consequently,
\[
K_+(f,v):=K(f,v)\cap \R^3_{\ge 0}
\]
is a proper convex cone (as an intersection of convex cones), as illustrated in \Cref{fig:3Dnonhypcones}.

Now consider the LEVI system \eqref{eq:levi} with the above matrix \(A\) and cone \(K_+(f,v)\). MATLAB simulations indicate that trajectories initialized in \(K_+(f,v)\) converge to the origin, illustrating stability with respect to the cone \(K_+(f,v)\), even though the linear system is unstable in the full space \(\R^3\). The cone and representative trajectories are shown in \Cref{fig:3Dnonhypcones}.
\begin{figure}[htbp]
\begin{center}
\begin{minipage}{0.50\linewidth}
  \begin{center}
    \includegraphics[scale=0.37]{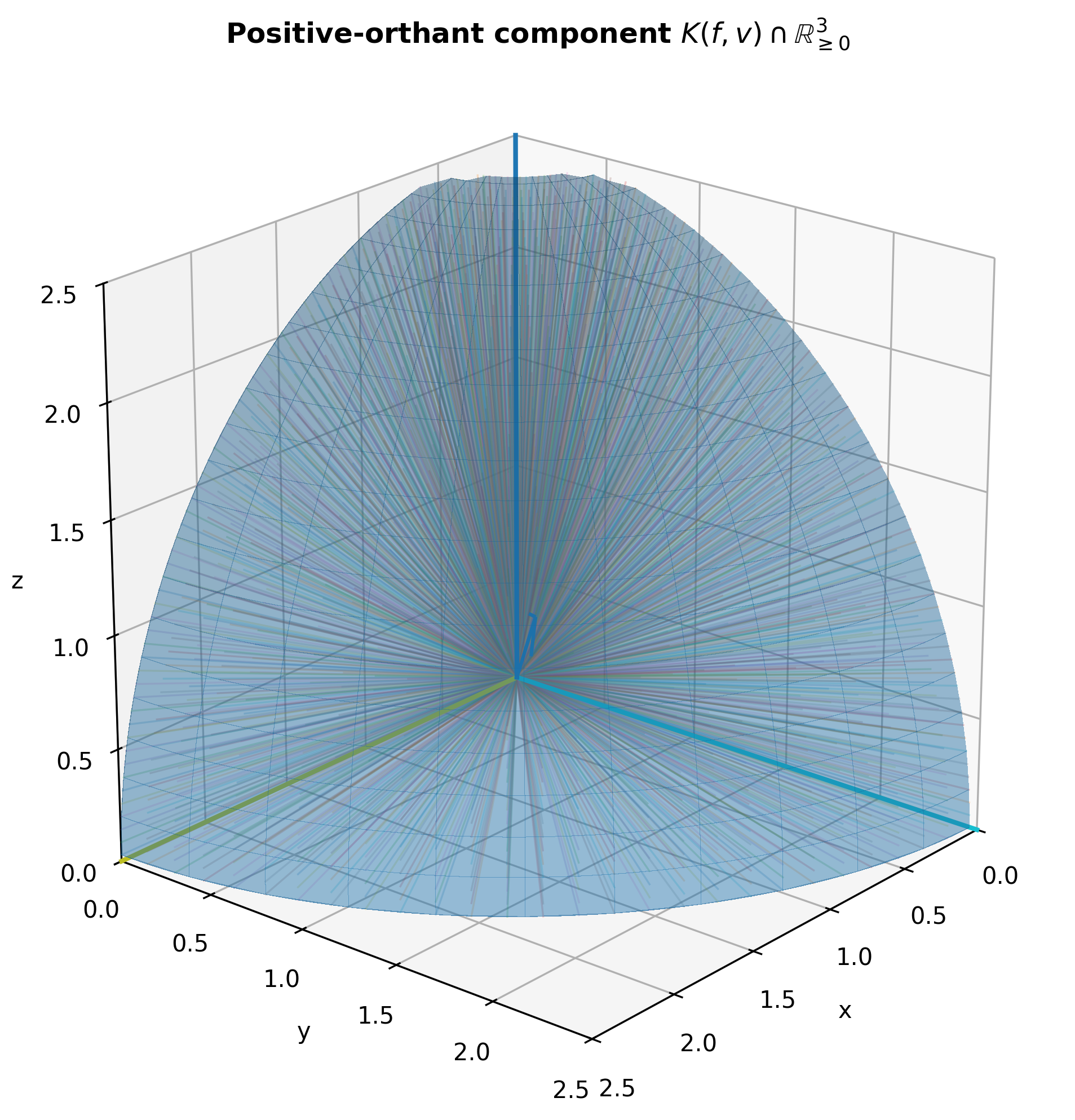}
    \caption{$\K(f,v)\cap\R^3_{\ge0}$.}
    \label{fig:3Dnonhypcones}
  \end{center}
\end{minipage}
\hfill
\begin{minipage}{0.4\linewidth}
  \begin{center}
    \includegraphics[scale=0.22]{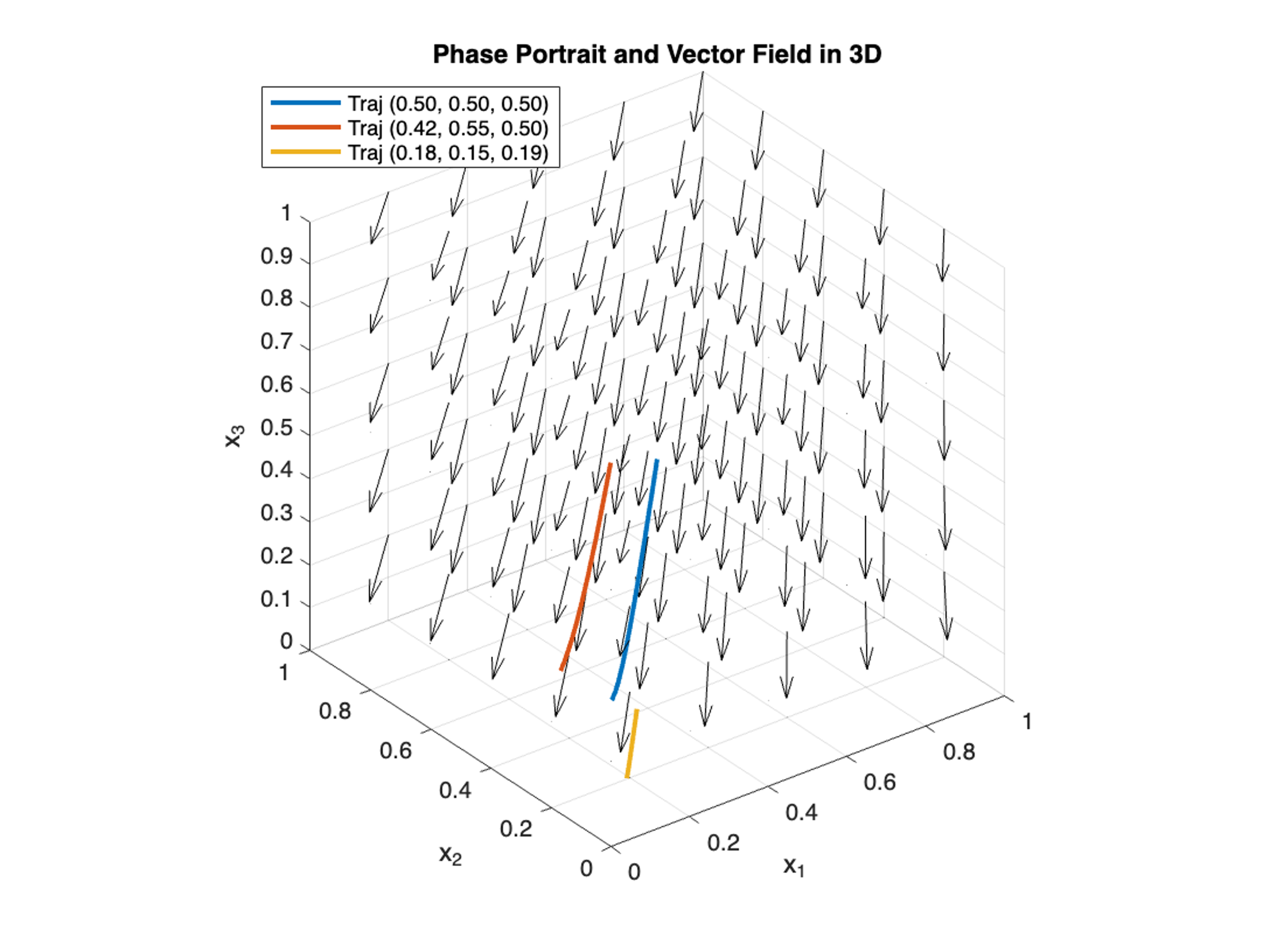}
    \caption{Stability of the origin w.r.t.\ $\K$.}
  \end{center}
\end{minipage}
\end{center}
\end{figure}
\end{example}
\begin{proposition}\label{prop:levi-hyperbolic-cone-stability}
Let \(A\in\R^{n\times n}\), \(A_s:=\tfrac12(A+A^{\mathsf T})\), and 
\[
f(x):=x^{\mathsf T}Ax=x^{\mathsf T}A_sx.
\]
Assume:
\begin{enumerate}[(i)]
    \item \(A_s\) has exactly one positive eigenvalue;
    \item \(A_s\) is \(\K\)-copositive (resp.\ strictly \(\K\)-copositive).
\end{enumerate}
then the trivial solution of the LEVI
system \eqref{eq:levi} is (asymptotically) stable w.r.t.\ $\K$.

\end{proposition}
\begin{proof} 
By (i), \(A_s=\tfrac12(A+A^{\mathsf T})\) has exactly one positive eigenvalue and since $A_s$ is copositive, so there exists \(v\in\R^n\) with  \(f(v)=v^{\mathsf T}A_s v>0\), so \(f(x)=x^{\mathsf T}Ax=x^{\mathsf T}A_sx\) is hyperbolic with respect to \(v\) and \(\K=K(f,v)\) is a convex cone. Hence $f$ is (strictly) \(\K\)-Lorentzian on \(\K=K(f,v)\). On the other hand, taking \(P=\tfrac12 I\) gives
\[
I-(P+P^{\mathsf T})=0
\quad\text{and}\quad
A^{\mathsf T}P+PA=A_s\in \mathcal P_{\K}\ \ (\text{resp.\ }\mathcal P_{\K}^{+}).
\]  Thus, \(A\in\mathcal L_{\K}\) (resp.\ \(A\in\mathcal L_{\K}^{++}\)).
Consequently, by \Cref{them:LEVIstability}, the trivial solution of \eqref{eq:levi} is stable (resp.\ asymptotically stable) with respect to \(\K\).
\end{proof}
\begin{theorem}\label{them:levi-hyperbolic-cone-stability}
If the quadratic form $f(x):=x^{\mathsf T}A_sx$ is $\K$ Lorentzian polynomial
then the trivial solution of the class of LEVI
systems \eqref{eq:levi} governed by any $A$ such that  $f(x):=x^{\mathsf T}A_sx=x^{\mathsf T}Ax$ is (asymptotically) stable w.r.t.\ $\K$.
\end{theorem}
\begin{proof}
Note that the symmetric representative $A_s$ is unique in $f(x):=x^{\mathsf T}A_sx=x^{\mathsf T}Ax$, but there is a class of matrices $\{A_s+N: \ \ \ N^{\mathsf T}=-N\}$ for which $f(x)=x^{\mathsf T}Ax$. Therefore by $\K$-Lorentzian property one can conclude the dynamics about a class of LEVI systems.
The proof follows from \Cref{prop:clcconstable} and \Cref{prop:levi-hyperbolic-cone-stability}.
\end{proof}
\begin{remark}
This criterion is particularly powerful because it reduces cone-based (asymptotic) stability of the LEVI system to a verifiable spectral condition on the symmetric part \(A_s=\tfrac12(A+A^{\mathsf T})\) (one positive eigenvalue) together with \(\K\)-copositivity, or equivalently \(\K\)-Lorentzianity of the induced quadratic form, without requiring \(A\) itself to be symmetric.
\end{remark}
\noindent \textbf{Acknowledgment}:
The author is grateful to June E. Huh for identifying issues in the earlier examples concerning the non-convexity of $K(f,v)$ and the LEVI system; these observations substantially improved the manuscript. During the preparation of this paper, the author also used ChatGPT to improve clarity and conciseness of the writing. All content was subsequently reviewed, revised, and verified by the author, who takes full responsibility for the final version of the article.

\bibliographystyle{alpha}
\bibliography{main}
\end{document}